\documentclass[12pt]{article}

\usepackage{latexsym,amssymb,amsmath}
\usepackage[dvips]{graphicx}
\newcommand{\C}{\mathbb C}
\newcommand{\Z}{\mathbb Z}
\newcommand{\N}{\mathbb N}

\newcommand{\Q}{\mathbb Q}

\newcommand{\sma}{\left(\begin{array}}
\newcommand{\fma}{\end{array}\right)}

\newtheorem{lem}{Lemma}[section]
\newtheorem{defn}[lem]{Definition}

\newtheorem{co}[lem]{Corollary}
\newtheorem{thm}[lem]{Theorem}
\newtheorem{prop}[lem]{Proposition}

\newenvironment{proof}{\textbf{Proof.}}{\newline\hspace*{\fill}{$\Box$}\\}

\begin{document}
\title{Balanced groups and graphs of groups with infinite cyclic edge groups}
\author{J.\,O.\,Button}

\newcommand{\Address}{{
  \bigskip
  \footnotesize

\textsc{Selwyn College, University of Cambridge,
Cambridge CB3 9DQ, UK}\par\nopagebreak
  \textit{E-mail address}: \texttt{j.o.button@dpmms.cam.ac.uk}
}}

\maketitle
\begin{abstract}
We give a necessary and sufficient condition for the fundamental group of
a finite graph of groups with infinite cyclic edge groups to be
acylindrically hyperbolic, from which it follows that a finitely
generated group splitting over $\Z$ cannot be simple. We also give a
necessary and sufficient condition (when the vertex groups are torsion
free) for the fundamental group to be balanced, where a group is said to
be balanced if $x^m$ conjugate to $x^n$ implies that $|m|=|n|$ for all
infinite order elements $x$.
\end{abstract}

\section{Introduction}
It is well known that a word hyperbolic group cannot contain a
Baumslag-Solitar subgroup $BS(m,n)$; indeed these have been called
``poison subgroups''. But whereas a group containing, say, $BS(2,3)$ is already
seen to be badly behaved because it contains a non residually finite
subgroup and so is itself non residually finite, this need not be the case
if a group contains $BS(1,1)=\Z^2$. On further reflection, the problems
seem to arise when we have subgroups $BS(m,n)$ for $|m|\neq |n|$ whereupon
we have an infinite order element $x$ with $x^m$ conjugate to $x^n$. This
phenomenon is already an obstruction in various settings, both group
theoretic and geometric. One way to think of this is that if we have a
``translation length'' function $\tau$ from a group $G$ to the reals which is
invariant on conjugacy classes and such that $\tau(g^n)=|n|\tau(g)$ then
any element $x$ with the property above must be sent to zero.

In \cite{wsox} a group was called unbalanced if there exists an infinite
order element $x$ with $x^m$ conjugate to $x^n$ but $|m|\neq |n|$. An
unbalanced group need not contain a Baumslag - Solitar subgroup in general
but there are conditions which ensure that it does. It was pointed out in this
paper that an unbalanced group cannot be subgroup separable and here we
also show that an unbalanced group cannot be a subgroup of $GL(n,\Z)$ for
any $n$. Indeed we begin in Section 2 by giving various examples of 
classes of groups that must necessarily be balanced and then in Section 3
we look at unbalanced groups. In order to do this we 
consider individual elements and say here that an infinite
order element $x$ with $x^m$ conjugate to $x^n$ but $|m|\neq |n|$ is unbalanced.
We then introduce two straightforward but useful
tools in order to determine whether a particular
infinite order element $a$ in an arbitrary group $G$
is balanced or not. The first is that of the
intersector, consisting of
all elements of $G$ that conjugate $\langle a\rangle$ to something
which intersects $\langle a\rangle$ non trivially. The point is that for
infinite cyclic subgroups this intersector is a subgroup, clearly containing
$\langle a\rangle$. The next tool is that of the modular map from
the intersector of $a$ to the non zero rationals, sending 
an element $g$ to $m/n$ if $ga^mg^{-1}=a^n$, which is easily verified to
be a well defined homomorphism and which will clearly detect 
unbalanced elements because it is equivalent to
the image of this homomorphism not lying in $\pm 1$.

Our first applications are to finite graphs of groups where all edge groups
are infinite cyclic but with little or no restriction on the vertex groups.
In Section 4 we look at when  the fundamental group of such a graph of
groups is acylindrically hyperbolic. A sufficient condition, with no 
restriction on edge or vertex groups, was given in  \cite{asos}. When all
edge groups are infinite cyclic but the vertex groups are arbitrary, we use
this theorem along with consideration of the intersectors in the vertex
groups of the inclusions of the edge subgroups to obtain necessary and
sufficient conditions for acylindrical hyperbolicity in Theorem \ref{mainac}.
An immediate application is that a finitely generated group that splits over
$\Z$ cannot be simple. 

In the rest of the paper we only deal with torsion free groups. In Section 5
we consider when the fundamental group of a
finite graph of groups is relatively hyperbolic. As this implies
acylindrical hyperbolicity, we in fact
give conditions which ensure that such a group is
not relatively hyperbolic with respect to any collection of proper subgroups.
For this we allow the vertex groups to be arbitrary torsion free groups
and the edge groups to be arbitrary non trivial subgroups of the vertex groups.
We then introduce the idea of the malnormal closure of a non trivial subgroup
of an arbitrary group and use this to obtain Corollary \ref{nonhypco} which
states that the fundamental group of a graph of groups as above is not
relatively hyperbolic if each vertex group has a cyclic subgroup whose
malnormal closure is the whole of this group. For instance this is the case
if each vertex group has an infinite soluble normal subgroup. 
 
It is not hard to see that for torsion free groups, being balanced is not
preserved in general
by taking amalgamated free products and HNN extensions, although
it is for free products. In the case where the edge groups are infinite cyclic
rather than trivial, we show in Theorem \ref{amal} using intersectors that
an amalgamated free product over $\Z$ of balanced groups
will still be balanced, and we give in Theorem \ref{hnnt} necessary and
sufficient conditions for an HNN extension over $\Z$. These are then used
in Theorem \ref{bal}
to answer the same question for the fundamental group of
a finite graph of groups with infinite cyclic edge groups and arbitrary
torsion free vertex groups. In order to do this we need to know when
two elements both lying in vertex groups have powers that are conjugate
in the fundamental group and this is dealt with in Section 7 by introducing
the idea of a conjugacy path. This can be thought of as an edge path in
the graph that records the fact that successive edge subgroups have 
conjugates in the intermediate vertex group that intersect non trivially,
without having to keep track throughout of the exact powers that occur.
This means that our necessary and sufficient condition
as to when the fundamental group of
a finite graph of groups with infinite cyclic edge groups and arbitrary
torsion free vertex groups is balanced in Theorem \ref{bal} is phrased
purely in terms of conjugacy paths that return to their starting element
and we prove that there are only finitely many of these paths that
need to be checked. One application is that for any such graph of groups
that fails this condition, the resulting fundamental group cannot be
subgroup separable or embeddable in $GL(n,\Z)$ for any $n$, no matter
how well behaved the vertex groups are.  

One would also like to say that such a graph of groups fails this
condition exactly when the fundamental group contains an unbalanced
Baumslag-Solitar subgroup, but we mentioned that some (torsion free)
unbalanced groups do not contain these as subgroups. To get round
this one could impose some reasonably wide ranging 
condition on the vertex groups,
such as being word hyperbolic or having cohomological dimension 2. We
define in Section 9 a condition on a group $G$ which is a
substantial generalisation of both of these, called the cohomological
condition, which states that any 2-generator subgroup 
$\langle a,b\rangle$ of $G$ where $\langle a\rangle$ meets $\langle b\rangle$
non trivially has cohomological dimension at most 2. We then show in
Corollary  \ref{bscoh} that if all vertex groups are torsion free, balanced
and satisfy this condition then the fundamental group is unbalanced
if and only if it contains an unbalanced Baumslag - Solitar subgroup.

Our final application in Section 10 is to the characterisation of
word hyperbolicity for a graph of groups with word hyperbolic vertex
groups and infinite cyclic edge groups. Now
the famous Bestvina - Feighn combination theorem in \cite{bf} and \cite{bf2}
gives a sufficient condition for a finite
graph of groups with word hyperbolic vertex
groups to have a word hyperbolic fundamental group and this need 
not require restriction to infinite cyclic edge groups. In the case where
they are infinite cyclic, these papers give necessary and sufficient
conditions for an amalgamated free product and HNN extension to be
word hyperbolic which can then be applied to determine whether or not the
fundamental group of a graph of groups with word hyperbolic vertex groups
and infinite cyclic edge groups has this property. However we have not
seen in the literature a characterisation of this which is given purely
in terms of information that can be directly read off from the graph
of groups without consideration of the action of the Bass - Serre tree.
(The closest equivalent to this we have come across is Appendix A in
\cite{ch} which allows one to establish whether the fundamental group
of a graph of groups with abelian edge groups has the CSA property.)
This can be done by using conjugacy paths as described above, as well as
what we call non
maximal paths which keep track of when an edge group embeds in a vertex
group as a non maximal cyclic subgroup.
Therefore it seemed
worthwhile recording in Theorem \ref{hyp} that a finite graph of groups
with torsion free word hyperbolic vertex groups and infinite cyclic
edge groups has word hyperbolic fundamental group exactly when the graph
of groups contains no complete conjugacy paths and no full non maximal paths.
Moreover there are only finitely many possibilities
for each of these two types of paths to occur.

\section{Examples of balanced groups}
The famous Baumslag-Solitar groups $BS(m,n)$ for $m,n\neq 0$ are given
by the presentation $\langle a,t|ta^mt^{-1}=a^n\rangle$ and are
HNN extensions where the base $\langle a\rangle$ and the associated
subgroups $\langle a^m\rangle$ and $\langle a^n\rangle$ are all infinite
cyclic, although note that they can also be expressed as an HNN extension
with base $\langle a,b|a^n=b^m\rangle$ and stable letter $t$ conjugating
$a$ to $b$; this latter base is not infinite cyclic unless one of
$|m|,|n|$ is equal to 1. 
Here we will divide them up into the following categories:\\
If $|m|=|n|$ then we call $BS(m,n)$ {\bf Euclidean}
(after \cite{hsws}), otherwise it is
non-Euclidean.\\
If one of $|m|$ or $|n|$ is equal to 1 then $BS(m,n)$ is {\bf soluble},
otherwise it is non-soluble.\\
Thus the Euclidean soluble Baumslag-Solitar groups are exactly $\Z^2$
and the Klein bottle group. Euclidean Baumslag-Solitar groups should be
regarded as generally very well behaved, for instance they are linear 
over $\Z$ and therefore
over $\C$ and are residually finite, they are subgroup separable (every finitely
generated subgroup is an intersection of finite index subgroups, thus again
they are residually finite) and we shall see shortly that they are
balanced, which is a definition from \cite{wsox}:
\begin{defn}
A group $G$ is called {\bf balanced} 
if for any element $x$ in $G$ of infinite order
we have that $x^m$ conjugate to $x^n$ implies that $|m|=|n|$.\\
Here we will also define:\\ 
A {\bf balanced element} in a group $G$
is an element $x$ in $G$ of infinite order
such that if we have $m,n\in\Z$
with $x^m$ conjugate to $x^n$ in $G$ then $|m|=|n|$. 
\end{defn} 
Thus a group is balanced if and only if all its elements of infinite order
are balanced.

The soluble Baumslag-Solitar groups (minus the two Euclidean ones) can in
turn be regarded as moderately well behaved: they are linear over $\C$,
indeed they embed in $SL(2,\C)$, 
thus are again residually finite, but they are not linear over $\Z$,
they are not subgroup separable - indeed this fails on the infinite cyclic 
subgroup $\langle a\rangle$ - and they are clearly not balanced. Meanwhile
the Baumslag-Solitar groups which are neither Euclidean nor soluble can
be regarded as very badly behaved indeed: they are famously not residually
finite and obviously not balanced, so any group having one of these as a
subgroup will also fail these two conditions.

We first quote some basic properties of balanced groups, then provide
a range of examples.
\begin{lem} \label{bsic} (\cite{wsox} Lemmas 4.13 and 4.14)\\
(i) If $G$ and $H$ are both balanced groups then so is
$G\times H$ and $G*H$.\\
(ii) If $G$ is a balanced group then so is any subgroup
$H$ of $G$. Conversely if $G$ contains a finite
index subgroup $H$ which is balanced then $G$ is balanced.
\end{lem}
However being balanced does not hold under extensions: for instance the
soluble Baumslag-Solitar group $BS(1,n)$ for $n>1$ has a torsion
free abelian normal subgroup with infinite cyclic quotient.

All word hyperbolic groups are balanced: indeed this fact
is established on the way to showing that a word hyperbolic group cannot
contain any Baumslag-Solitar subgroup. As for variations and generalisations
on this result, it will be immediate from Lemma \ref{hypr} later in this paper
that if a group is torsion free and hyperbolic relative to a collection of
balanced subgroups then it is also balanced. Also
groups that act properly and cocompactly on a CAT(0) space are
balanced (thus the Burger-Mozes examples are torsion free simple groups
that are balanced).
A torsion free word hyperbolic group is CSA (conjugate
separated abelian), meaning that the centraliser of any non trivial
element is abelian and malnormal: CSA groups are also clearly balanced
(indeed we must have $m=n$ here). Another large class of balanced
groups is the class of 3-manifold groups: interestingly
this was shown in \cite{krop3} well before the proof of Geometrization.

Abelian groups are obviously balanced but what about
replacing abelian by nilpotent/polycyclic/soluble? Once again
$BS(1,n)$ is a soluble counterexample but otherwise this holds by
the following which is \cite{wsox} Lemma 4.12.
\begin{lem}
If $G$ is subgroup separable then it is  
balanced; indeed if $gx^mg^{-1}=x^n$ holds
in $G$ for $x$ of infinite order and
$|m|\neq |n|$ then one of the infinite cyclic subgroups
$\langle x^m\rangle$ or $\langle x^n\rangle$ 
is not separable in  $G$.
\end{lem} 
As it is a result of Mal'cev that virtually polycyclic groups are
subgroup separable, we have that they are also balanced. 

Further examples can be obtained by using residual properties.
\begin{prop} \label{res}
If $G$ is a group which is residually (torsion free balanced), that
is for all non trivial $g\in G$ we have a homomorphism onto a
torsion free balanced group with $g$ not in the kernel (so $G$ is
itself torsion free), then $G$ is balanced.
\end{prop}
\begin{proof}
If $gx^mg^{-1}=x^n$ holds in $G$ for $m\neq n$ then the commutator
$[g,x]$ is non trivial, so take a homomorphism where $[g,x]$ does not
vanish in the image, thus neither does $g$ and $x$. As our relation
still holds but we are now in a torsion free balanced group, we find
that $|m|=|n|$.
\end{proof}
The most common families coming under Proposition \ref{res} are residually
free groups, or more generally residually (torsion free nilpotent)
groups. In particular this is one way to see that
all limit groups and all right angled Artin groups are balanced.

As $F\times\Z$ is balanced for $F$ a free group, this tells us
that the Euclidean Baumslag-Solitar groups are balanced because they
are virtually $F\times\Z$. We have a variation on this:
\begin{prop}
A group $G$ that is virtually (free by cyclic) is balanced.
\end{prop}
\begin{proof}
We can assume that $G$ is free by cyclic by Lemma \ref{bsic} (ii)
which here will mean $G=F\rtimes\Z$ for $F$ free but not necessarily
finitely generated.

If $gx^mg^{-1}=x^n$ for $m\neq n$ then $x$ must lie in the kernel
of the projection to $\Z$, thus it and $gxg^{-1}$ lie in $F$. 
But $(gxg^{-1})^m=x^n$ implies that 
$x$ and $gxg^{-1}$ are both contained
in the same maximal cyclic subgroup $\langle z\rangle$ of $F$, so that
$x=z^k$ and $gxg^{-1}=z^l$ for some $k,l$ where $nk=ml$. But now
$(gzg^{-1})^{mk}=z^{nk}$ implies that $gzg^{-1}$ also lies in 
$\langle z\rangle$, so if $gzg^{-1}=z^j$ then $j=\pm 1$ (by also
considering $g^{-1}zg$) and
$z^{nk}=x^n=(gxg^{-1})^m=(gzg^{-1})^{mk}=z^{\pm mk}$ so $|m|=|n|$.
\end{proof}

Our last set of examples involve linear groups, but not over $\C$ as
we know that $BS(1,n)$ lies in $GL(2,\C)$, in fact in $GL(2,\Q)$ and
even in $SL(2,\Q)$ if $n$ is a square, but is not balanced for $n>1$.
In fact we consider subgroups of $GL(n,\Z)$, which is of interest
because right angled Artin groups and
``cubulated'' groups in the sense of Wise embed in $GL(n,\Z)$. One might
ask for a group theoretic obstruction to being linear over $\Z$ 
(excluding those that are obstructions to being linear over $\C$ such
as being finitely generated but not
residually finite). We only know of one, again due to Malce'ev,
which is that every soluble subgroup must be polycyclic. Thus we might
ask whether failure to be balanced is also an obstruction 
and indeed this is true.
\begin{thm}
If $gx^mg^{-1}=x^n$ holds in a group $G$ where $|m|\neq |n|$ and $x$ 
has infinite order then $G$ is not linear over $\Z$.
\end{thm}
The proof is given in Section 3, although one can also establish this by
using linear algebra over $\C$. Oddly in both cases
linearity over $\Z$ is never used other than to quote Malce'ev's result. 

\section{The intersector}
In the last section we considered balanced groups. We now go into
a little more detail and look at balanced elements.

If $H$ is a subgroup of $G$
then we can consider the set of elements $g\in G$ 
such that $gHg^{-1}\cap H\neq\{e\}$. In Geometric Group Theory
these have been called the intersecting conjugates of $H$ (as in \cite{wsrr}
p26) although the set is also known as the generalised normaliser of $H$ in 
Combinatorial Group Theory, as defined in \cite{ks}. 
Now this set 
will not form a subgroup in general but a basic yet fundamental
observation here is that
it will if $H$ is infinite cyclic.
\begin{defn}
If $a$ is an element of infinite order in a group $G$ then
we define the {\bf intersector $I_G(a)$ of $a$ in $G$} to be
\[\{g \in G:g\langle a\rangle g^{-1}\cap\langle a\rangle\neq\{e\}\}.\]
\end{defn}
\begin{lem} \label{lem1}
(i) $I_G(a)$ is a subgroup of $G$ containing $\langle a\rangle$.\\
(ii) If $a^k=b^l$ for some $k,l\neq 0$ then $I_G(a)=I_G(b)$.\\
(iii) For any $g\in G$ we have $I_G(gag^{-1})=gI_G(a)g^{-1}$.
\end{lem}
\begin{proof}
(i) The identity and inverses are clearly in $I_G(a)$ so suppose
$ga^ig^{-1}=a^j$ and $ha^kh^{-1}=a^l$ where all of $i,j,k,l$ are non zero
then $gh a^{ik}(gh)^{-1}=a^{jl}$.\\
(ii) If $g\in I_G(a)$, so that
$ga^ig^{-1}=a^j$ for $i,j\neq 0$, then $gb^{il}g^{-1}=ga^{ik}g^{-1}=
a^{jk}=b^{jl}$ and we can now swap $a$ and $b$.\\
(iii) Conjugation by $g$ is an automorphism of $G$ and the definition
of $I_G(a)$ is purely group theoretic.
\end{proof}
Thus $I_G(a)=\langle a\rangle$ if and only if $\langle a\rangle$ is a 
{\it malnormal} subgroup of $G$. Indeed Lemma \ref{lem1} (i) was already
in \cite{ks}, which applied it to circumstances where $\langle a\rangle$
was malnormal or close to being malnormal in $G$ but here we are
interested in the general setting. Also $I_G(a)$ is equal to the commensurator
of $\langle a\rangle$ in $G$.

We now introduce the idea of the modular homomorphism, which can be thought
of as a variation on the concept of the same name which is defined for
generalised Baumslag-Solitar groups. Here we can work with an arbitrary
group but we only obtain a ``local'' version.
\begin{defn}
Given a group $G$ and an element $a\in G$ of infinite order, the
{\bf modular homomorphism} $\Delta_a^G$ of $a$ in $G$ (or $\Delta_a$
when the group is clear)
is the map from the intersector $I_G(a)$ to
the multiplicative non zero rational numbers $\Q^*$ defined as follows:
if $g\in I_G(a)$ so that $ga^mg^{-1}=a^n$ for some $m,n\in\Z\setminus\{0\}$
then we set $\Delta_a^G(g)=m/n \in \Q^*$.
\end{defn}
It is not hard to see that $\Delta_a^G$ is well defined (because $a$ has
infinite order) and is a homomorphism (this homomorphism 
was noted in \cite{krop3} although with domain restricted to
$\langle g,a\rangle$ for a given $g\in I_G(a)$). Note also that if $g$ has
finite order then $\Delta_a^G(g)=\pm 1$.
\begin{defn}
The {\bf modulus} of $a\in G$ is the image of $\Delta_a^G$ in $\Q^*$ and
$a$ is called a {\bf unimodular} element of $G$ if its modulus is
contained in $\{\pm 1\}$.
\end{defn} 
Note that if we have two elements $a,b\in G$ of infinite order
with a
non trivial power of $a$ conjugate to a power of $b$ then $a$ and $b$
have the same modulus by repeated use of
Lemma \ref{lem1} (ii) and (iii).  
Also we see that a group is balanced if and only if all
its elements of infinite order are unimodular. 

In the previous section we gave many examples of balanced groups, but the only
unbalanced groups mentioned were the non Euclidean 
Baumslag-Solitar groups. Of course any group with one
of these as a subgroup would also be unbalanced, but we have not yet seen any
unbalanced groups which do not contain any Baumslag-Solitar group.\\
\hfill\\
{\bf Example}
Let $r,s$ be non zero coprime integers where $|r|$ and $|s|$ are not both 1  
and consider the ring $\Z[1/rs]$ but considered as a torsion free abelian
group $A$, here written additively, which is locally cyclic as it is a
subgroup of $\Q$. Let us now form
the semidirect product $G_{r,s}=G=A\rtimes\Z$ where the
generator $t$ of $\Z$ acts on $a\in A$ by conjugating $a$ to $(s/r)a$,
which is a metabelian group. Now any element of
$G_{r,s}$ can be written in the form $x=(a,t^k)$, with $G_{r,s}$ generated
by $(1,e)$ and $(0,t)$,  and
if $k\neq 0$ then $x$ is a balanced element (for instance we can
project into $\Z$), thus we see that the set of unbalanced elements in
$G$ is exactly $A\setminus\{0\}$, with the modulus of each element
equal to $\{(r/s)^k:k\in\Z\}$. 
We then have that if neither $|r|$ nor $|s|$ is 1 then $G=G_{r,s}$ does
not contain any Baumslag-Solitar subgroup:
first if $G$ contains $BS(m,n)$ for $|m|\neq |n|$
then neither $|m|$ nor $|n|$ can equal 1 because the only integer
contained in the modulus of any $a\in A$
is $1$. Also $A$ does not contain $\Z^2$ so if $G$ contained
such a subgroup $H\cong\Z^2$ then $H$ would have
non trivial intersection with $A$ but would also have non trivial image
under projection to the $\Z$ factor, so would contain an element of the
form $(a,t^k)$ for $k\neq 0$ which does not commute with anything in
$A\setminus \{0\}$. Thus $G$ can only contain 
non soluble Baumslag Solitar groups but it is soluble.

If $N$ is the normal closure of $a$ in 
$BS(r,s)=\langle a,t|ta^rt^{-1}=a^s\rangle$
and $N'$ the commutator subgroup of $N$
then it is well known that $G_{r,s}$ is isomorphic to $BS(r,s)/N'$,
with $a$ mapping to $(1,e)$ and $t$ to $(0,1)$.
Here if one of $|r|$ or $|s|=1$ then $N'$ is trivial.
Other interesting properties of this correspondence between
$BS(r,s)$ and $G_{r,s}$
when neither $|r|$ nor $|s|$ equal 1 are that $G_{r,s}$
is not finitely presented (\cite{bast}) and that the finite residual
$R$ of $BS(r,s)$ (the intersection of all finite index subgroups) is
equal to $N'$ \cite{mld}. 
These observations can be used to give fairly general
circumstances under which a group which is not balanced
will contain a non Euclidean Baumslag Solitar subgroup.
\begin{lem} \label{homo} If we have a surjective homomorphism
of $G_{r,s}$ in which the image of $(1,e)$ has infinite
order then this is an isomorphism.
\end{lem}
\begin{proof}
If $(a,e)$ is in the kernel for $a\in A\setminus \{0\}$ then so is
$(n,e)$ for some non zero integer $n$. Otherwise we have some element
$(\cdot,t^k)$ in the kernel for $k\neq 0$, 
but as this element conjugates
$(x,e)$ into $((s/r)^kx,e)$ and $(s/r)^k$ is never 1 here,
we again have $a\in A\setminus \{0\}$ such that
$(a,e)$ is in the kernel too, thus so is some power of $(1,e)$.
\end{proof}
\begin{prop} \label{stff}
Let $G$ be an unbalanced group.\\
(i) If some modulus contains an integer not equal to $\pm 1$ then
$G$ must contain a
Baumslag-Solitar subgroup $BS(1,n)$ for $|n|>1$.\\
(ii) If $G$ is residually finite then $G$ must contain a subgroup
of the form $G_{r,s}$.\\
(iii) If $G$ is residually finite and coherent (meaning that every
finitely generated subgroup is finitely presented)
then $G$ must contain a 
Baumslag-Solitar subgroup $BS(1,n)$ for $|n|>1$.
\end{prop}
\begin{proof}
If (i) holds then we have infinite order elements $g,a\in G$
with $ga^mg^{-1}=a^n$ where $m=ln$ for $|l|>1$. Thus on replacing $a$ with
$a^n$ we have $g^{-1}ag=a^l$ so that $\langle a,g\rangle$ is a homomorphic
image of $BS(1,l)$. But as $G_{1,l}=BS(1,l)$ and $a$ has infinite order
we have by Lemma \ref{homo} that $\langle a,g\rangle$ is isomorphic to
$BS(1,l)$.

If now $G$ is residually finite but unbalanced, so we have $ga^rg^{-1}=a^s$
for $|r|\neq |s|$ and $|r|,|s|$ coprime without loss of generality then
again $\langle a,g\rangle$ is a homomorphic image of $BS(r,s)$. But it is
a residually finite image so the homomorphism  must factor
through $BS(r,s)/R=G_{r,s}$. Then we can again apply Lemma \ref{homo}.

If however $G$ is also coherent then $G_{r,s}$ cannot be a subgroup of
$G$ unless one of $|r|$ or $|s|$ is 1 (so say $r=1$), in which case 
being unbalanced implies that we have infinite order
elements $a,g\in G$ with $gag^{-1}=a^s$ for $|s|>1$, thus $s$ is
in the modulus of the element $a$ of $G$. In particular in a residually
finite coherent group, every modulus consists only of integers and their
reciprocals.
\end{proof}

Another application to unbalanced groups is that of CT (commutative
transitive) groups, such as torsion free subgroups of $SL(2,\C)$. These
are generalisations of CSA groups, with the latter always balanced 
so we might expect CT groups to be as well. However the
fact that $G_{r,s}$ is a subgroup of $SL(2,\C)$ via the embedding
\[(1,e)\mapsto\sma{cc}1&1\\0&1\fma,\qquad (0,t)\mapsto
\sma{rr}\sqrt{\frac{s}{r}}&0\\0&\sqrt{\frac{r}{s}}\fma\]
(which is injective by Lemma \ref{homo} and so is CT) 
provides immediate counterexamples - though in fact 
essentially the only counterexamples.
\begin{prop} \label{ctid}
If $G$ is a CT group that is not balanced then it contains
an isomorphic copy of $G_{r,s}$ for $|r|\neq |s|$.
\end{prop}
\begin{proof}
On again taking $ga^rg^{-1}=a^s$ for $g,a$ elements of infinite order and
$r,s$ coprime but $|r|\neq |s|$, we have that $a$ commutes
with $ga^rg^{-1}$ which commutes with $gag^{-1}$, so $gag^{-1}$ commutes with
$a$, whereupon $g^2ag^{-2}$ commutes with $gag^{-1}$ and thus with $a$, and
so on. Thus the normal closure of $a$ in $\langle a,g\rangle$ is abelian,
so that $\langle a,g\rangle$ is a homomorphic image of 
$BS(r,s)/N'=G_{r,s}$ and so is equal to $G_{r,s}$ by Lemma \ref{homo}.
\end{proof}

Finally we can provide the proof of our result left over from the last section.
\begin{thm}
A subgroup of $GL(n,\Z)$ is balanced.
\end{thm}
\begin{proof}
On taking the usual $\langle a,g\rangle$ with $ga^rg^{-1}=a^s$ for
$r,s$ coprime and $|r|\neq |s|$, we have that $\langle a,g\rangle$ is a
finitely generated linear group, thus residually finite so Proposition
\ref{stff} (ii) applies. But $G_{r,s}$ is then a soluble subgroup of
$GL(n,\Z)$ so must be polycyclic which is a contradiction.
\end{proof}

\section[Acylindrical hyperbolicity of graphs of groups]{Acylindrical 
hyperbolicity of graphs of groups with infinite
cyclic edge groups}
We now turn to a topic which, on the face of it, seems to have little to
do with balanced groups. 
In \cite{asos} a subgroup $H$ of a group $G$ is called {\it weakly malnormal}
if there is $g\in G$ such that $gHg^{-1}\cap H$ is finite, and 
{\it $s$-normal} otherwise. Thus if we take $H=\langle a\rangle$ to be 
infinite cyclic, we see that $\langle a\rangle$ being 
$s$-normal/weakly malnormal is
equivalent to $I_G(a)$ being equal/not equal to $G$.
In that paper this concept was introduced in the context of acyclindrical 
hyperbolicity,
with a group being acylindrically hyperbolic implying that it is
$SQ$-universal and in particular is not a simple group. The paper gave
conditions under which the fundamental group of a graph of groups
is acylindrically hyperbolic and then provided related applications.
Here we will stick to the case where all edge groups are infinite cyclic,
whereupon we can use intersectors to determine exactly when a finite
graph of groups with infinite cyclic edge groups is acylindrically
hyperbolic.

The relevance of $s$-normal subgroups in this setting
is twofold: first if a group $G$ is acylindrically hyperbolic then any 
$s$-normal subgroup of $G$
is itself acylindrically hyperbolic: in particular it cannot be a cyclic
subgroup. Also in Section 4 of \cite{asos} we have
sufficient conditions for acylindrical hyperbolicity which we now describe:
given a graph of groups $G(\Gamma)$ with connected graph $\Gamma$ and
fundamental group $G$, an edge $e$ is called good if both edge inclusions into
the vertex groups at either end of $e$ give rise to proper subgroups, otherwise
it is bad. A reducible edge is a bad edge which is not a self loop.
Given a finite graph of groups, we can contract the reducible edges one by
one until none are left. This process does not affect the fundamental group
$G$ and the new vertex groups will form a subset of the original vertex groups. 
It could be that we are left with a single vertex and no edges, in which case
we say that the graph of groups $G(\Gamma)$ was trivial with $G$ equal to
the remaining vertex group. We then have:
\begin{thm} \label{aomn} (\cite{asos} Theorem 4.17)
Suppose that $G(\Gamma)$ is a finite reduced graph of groups which is 
non trivial and which is not just a single vertex with a single bad edge.
If there are edges $e,f$ of $\Gamma$ (not necessarily distinct) 
with edge groups $G_e,G_f$ and an element $g\in G$ such that
$G_f\cap gG_eg^{-1}$ is finite then $G$ is either virtually cyclic or
acylindrically hyperbolic.
\end{thm}
Consequently 
we might hope that if we have a graph of groups with infinite cyclic
edge groups then the existence of $s$-normal cyclic subgroups
characterises whether the fundamental group of this graph of groups
is acylindrically hyperbolic, which indeed turns out to be the case.
In fact we can characterise acylindrical hyperbolicity directly from
the graph of groups.
\begin{thm} \label{mainac}
Suppose that $G$ is the fundamental group of a non trivial finite 
reduced graph of groups $G(\Gamma)$ where
all edge groups are infinite cyclic, but with no further restriction
on the vertex groups. Then
$G$ is acylindrically hyperbolic unless both of the following two
conditions hold:\\
(i) At each vertex $v\in\Gamma$ with vertex group $G_v$, the 
intersection of all the inclusions of the edge
groups incident at $v$ is a non trivial subgroup of $G_v$.\\
(ii) If (i) holds, so that at each vertex $v$ we can
let $g_v\neq e$ be a generator
of this intersection, then we have $I_{G_v}(g_v)=G_v$.\\
In this case $\langle g_v\rangle$ is $s$-normal in $G$ 
for any $v\in\Gamma$ and
so $G$ is not acylindrically hyperbolic.
\end{thm} 
\begin{proof}
First let us assume that (i) fails. As a finite intersection of infinite
cyclic 
subgroups is trivial if and only if all pairwise intersections are trivial,
we will have a vertex $v\in\Gamma$
with the images in the vertex group $G_v$
under inclusion of two edge groups being
$\langle a\rangle$ and $\langle b\rangle$ say 
where $\langle a\rangle\cap\langle b\rangle$ is trivial.
Consequently $G_v$ and
hence $G$ will not be virtually cyclic, which also means 
we do not have a single
vertex with a single bad edge as the edge subgroups are infinite cyclic,
thus we can immediately apply Theorem 
\ref{aomn}.

Now we suppose that (i) holds but at some vertex $v$ we have the generator
$g_v\neq e$ of the intersection of the edge group inclusions has 
$I_{G_v}(g_v)\neq G_v$. Then $I_{G_v}(g_v)=I_{G_v}(a)$ for $a$ a generator
of one of the edge group inclusions into $G_v$. Thus
there is an element $x\in G_v\setminus I_{G_v}(a)$, meaning that
$\langle a\rangle\cap x\langle a\rangle x^{-1}$ is trivial and $G_v$ is
not virtually cyclic. So we again apply Theorem \ref{aomn}.

Conversely if (i) and (ii) both hold then there exists an element $g\in G$ 
which is a power of $g_v$ for every $v\in\Gamma$. To see this, first take
a maximal tree $T$ in $\Gamma$ and form the group $G_T$ from this tree using
amalgamated free products. We argue by induction on the number of edges:
suppose that $T=T_0\cup\{e_l\}$ where $e_l$ is a leaf edge and $T_0$ has
$e_l$ removed, with $G_{T_0}$ the fundamental group of this graph of groups.
Suppose that $g_0\in G_{T_0}$ is a power of every $g_v$ for $v\in T_0$.
Then on adding the edge $e_l$ with new vertex $v_1$ to $T_0$ at the vertex
$v_0\in T_0$, the edge group of $e_l$ provides the inclusion $\langle a\rangle$
into $G_{v_0}$ and $G_{v_1}$. Now $\langle g_{v_0}\rangle\leq\langle a\rangle$
with $g_{v_0}\neq e$ by condition (i) and we are supposing that $g_0$ is a
power of $g_{v_0}$, hence a power of $a$. But at the vertex $v_1$ we have that
$g_{v_1}$ is also a power of $a$, telling us that an appropriate power $g$ of
$g_0$ is also a power of $g_{v_1}$, and thus of every $g_v$ for $v\in T$. This
then confirms our claim for the group $G_T$, but this embeds in $G$ on 
introducing the stable letters for the edges in $\Gamma\setminus T$ and
forming the HNN extensions. Moreover $G$ is generated by its vertex groups
$G_v$ and those stable letters, so that we certainly have all $G_v$ in
$I_G(g)$ by condition (ii) and Lemma \ref{lem1} (ii). Finally each edge in
$\Gamma\setminus T$ provides a stable letter $t$ and edge group inclusions
$\langle a\rangle$ into some $G_v$ and $\langle b\rangle$ into some $G_w$
(possibly $v=w$) where $tat^{-1}=b$. But our element $g$ is  a power of both
$a$ and $b$, so that $t$ is in $I_G(g)$ as well and hence $I_G(g)=G$. Thus
as $g$ is a power of any $g_v$, we have that $I_G(g_v)=G$ too and hence
$G$ is not acylindrically hyperbolic. 
\end{proof}

We can now use this to show that a finitely generated
simple group cannot split over $\Z$.
It is known that a fundamental group of a non trivial finite graph of
groups with all edge groups finite cannot be simple by \cite{loss}. 
\begin{thm} \label{splz}
If $G(\Gamma)$ is a non trivial graph of groups with finitely generated
fundamental group $G$ and where all edge groups are
infinite cyclic 
then $G$ is either:\\
(i) acylindrically hyperbolic or\\
(ii) has a homomorphism onto $\Z$ or\\
(iii) has an infinite cyclic normal subgroup.\\
In particular $G$ is never simple.\\
\end{thm}
\begin{proof}
As $G$ is finitely generated,
we can assume that $\Gamma$ is a finite graph,
and then as above we can assume
that $G(\Gamma)$ is reduced. It is well known that if $\Gamma$ is not
a tree then $G$ surjects to $\Z$ as $\pi_1(\Gamma)$ is non trivial.
We now apply Theorem \ref{mainac} to obtain acyclindrical hyperbolicity
of $G$ unless both (i) and (ii) hold, in which case
we have our infinite order element
$g\in G$ which lies in every vertex group $G_v$ and such that 
$I_{G_v}(g)=G_v$, but $I_G(g)$ is a subgroup of $G$ containing all $G_v$
and thus is equal to $G$ as $\Gamma$ is a tree.
It is here that the ideas of balanced groups and elements
reenter: first suppose that $g$ is unimodular in every vertex group. Then
$\Delta_g^{G_v}(G_v)$ is contained in the subgroup $\pm 1$ of $\Q^*$ and
therefore so is $\Delta_g^G(G)$. Thus for all $x\in I_G(g)=G$ we can find
an integer $k>0$ such that $xg^kx^{-1}=g^{\pm k}$. Here $k$ depends on $x$
but as we have a finite generating set for $G$ we can find a common power $P$
that
works for all of this set and hence for all of $G$, thus $\langle g^P\rangle$
is normal in $G$. 

Now suppose that we have $v\in\Gamma$ such that $\Delta_g^{G_v}$ is not 
contained in $\pm 1$. We borrow a trick from \cite{krop3}, which is that
$|\Delta_g^{G_v}|$ provides a homomorphism from $G_v$ to the positive rationals
which is non trivial, thus the image is an infinite torsion free abelian
group and $g$ is in the kernel. However this means that $g_v$ is too and
hence also $g_e$ for $g_e$ the generator of any edge group with edge $e$
incident at $v$. Thus on sending all other vertex groups to the identity,
we extend the domain of this homomorphism to all of $G$, which being
finitely generated means that the infinite torsion free abelian image
of $G$ must be $\Z^n$.
\end{proof} 

\section{Absence of relative hyperbolicity}
For the remainder of the paper we will consider only torsion free groups,
so as to allow for clean statements that do not require consideration
of many cases in the corresponding proofs. Henceforth element will mean
a group element of infinite order and power will mean a non zero power,
thus for instance saying that elements $x,y$ in $G$ have conjugate powers is
a shorthand for saying that for all $i,j>0$ we have $x^i,y^j\neq e$ but
there exists $g\in G$ and non zero integers
$m,n$ such that $gx^mg^{-1}=y^n$. 

We now examine whether the fundamental group of a graph of groups 
is relatively hyperbolic with respect to a collection of proper subgroups. 
As this would imply acylindrical hyperbolicity anyway (at least if the
group is not virtually cyclic), our emphasis will be on finding conditions
that ensure that the groups considered in the previous section are not
hyperbolic with respect to any collection
of proper subgroups. For this we adopt the method from 
\cite{bkrp}
Section 4, which itself borrows from $\cite{aas}$. We first summarise the
facts we need about relatively hyperbolic groups, all of which come from
\cite{os04}. We suppose that $G$ is hyperbolic relative to a collection
of proper subgroups $H_1,\ldots, H_l$, the peripheral subgroups, and we say
that $g\in G$ is hyperbolic if $g$ is not conjugate into a peripheral subgroup.
We also assume here for these statements that $G$ is torsion free.
We then have:\\
\textbullet (\cite{os04} Theorem 4.19) If $g\in G$ is hyperbolic then
the centraliser $C_G(g)$ is strongly relatively quasiconvex in $G$. (Here
strongly means that its intersection with any conjugate of a peripheral
subgroup is finite, as in \cite{os04} Definition 4.11.)\\
\textbullet (\cite{os04} Theorem 4.16) A strongly relatively quasiconvex
subgroup of $G$ is word hyperbolic.\\
\textbullet (\cite{os04} Corollary 4.21) If $g$ is a hyperbolic element
in $G$ and we have $t\in G$ with $tg^kt^{-1}=g^l$ then
$|k|=|l|$.\\
\textbullet (\cite{os04} Theorem 1.4)
Any $H_i$ is malnormal, so that if there is
$g\in G$ with $H_i\cap gH_ig^{-1}$ non trivial then $g\in H_i$. Moreover
if there is $g\in G$ with $H_i\cap gH_jg^{-1}$ non trivial then $i=j$.

We know that in a torsion free word hyperbolic group $G$ (where we can just
take the single peripheral subgroup $\{e\}$) the centraliser $C_G(g)$ of any
non identity element $g$ is a maximal cyclic subgroup of $G$, but we can
see that this also holds for the intersector. Indeed we have:
\begin{lem} \label{hypr}
For any non identity $g$ in a group $G$ which is torsion free and hyperbolic
relative to a collection of proper subgroups, either $g$ is hyperbolic in
which case $I_G(g)$ is a maximal cyclic subgroup of $G$ or $I_G(g)$ is
conjugate into a peripheral subgroup.
\end{lem} 
\begin{proof} If $g$ is hyperbolic then $g^k$ is also, as if $g^k\in 
\gamma P\gamma^{-1}$
for $P$ a peripheral subgroup and $\gamma\in G$ then 
$g^k\in g(\gamma P\gamma^{-1})g^{-1}\cap 
\gamma P\gamma^{-1}$, so $g\in \gamma P\gamma^{-1}$ by malnormality.
But the first two points above say that $C_G(g)$ is a word hyperbolic group 
which here will also be torsion free with an infinite centre, 
so it must be infinite cyclic (and maximal in $G$
as it is a centraliser), say 
$C_G(g)=\langle h\rangle$. Now let us take $t\in I_G(g)$, so
that we have $tg^kt^{-1}=g^l$ for $|k|=|l|$, by the above. If $l=k$ then we
see that $t$ commutes with $g^k$, which has the same centraliser as $g$
(because $g^k$ is hyperbolic with $C_G(g)\leq C_G(g^k)$ and these are
maximal cyclic subgroups), so that $t$ is in $C_G(g)$ too. If $l=-k$ then
$t^2$ commutes with $g^k$ so say $t^2=h^i$ and $g=h^j$, whereupon
$(tg^kt^{-1})^i=g^{-ik}$ implies that
\[tt^{2jk}t^{-1}=th^{ijk}t^{-1}=g^{-ik}=h^{-ijk}=t^{-2jk}\]
so that $t$ is a torsion element which is a contradiction.

Now suppose that $g$ lies in 
$\gamma P\gamma^{-1}$ for $P$ a peripheral subgroup and
$\gamma\in G$. If $t\in I_G(g)$ so that $tg^it^{-1}=g^j$ then we have 
$g^j\in t(\gamma P\gamma^{-1})t^{-1}\cap 
\gamma P\gamma^{-1}$ so by malnormality as before we
obtain $t\in \gamma P\gamma^{-1}$.
\end{proof}

We can now prove under very general circumstances that graphs of groups are
not relatively hyperbolic. Of course the condition in the theorem
below on non trivial edge 
groups is needed, otherwise we can obtain free products which will be
relatively hyperbolic with respect to the factors.
\begin{thm} \label{nonhyp}
Let $G(\Gamma)$ be a finite graph of groups where
each vertex group is non trivial and torsion free 
but with no further restriction on the edge groups other than they 
are all non trivial. Suppose that $G$ is relatively hyperbolic with respect to
the collection of subgroups $H_1,\ldots ,H_l$. Suppose further that for all
vertices $v\in\Gamma$, there is a peripheral subgroup $H_{i(v)}$ such that
the vertex group $G_v$ is conjugate in $G$ into $H_{i(v)}$. Then
some peripheral subgroup is equal to $G$,
so that $G$ is not hyperbolic relative to any
collection of proper subgroups.
\end{thm}
\begin{proof} 
On picking a maximal tree in $\Gamma$ and any vertex $v$, 
there is $\gamma\in G$ and a peripheral subgroup 
$P$ such that $G_v\subseteq \gamma P\gamma^{-1}$.
But for any $w$ adjacent to $v$ in $T$ we similarly have $\delta\in G$ and a
peripheral subgroup $Q$ such that $G_w\subseteq \delta Q\delta^{-1}$ . But
$G_v\cap G_w$ is non trivial as it contains this edge group so 
$\gamma P\gamma^{-1}=
\delta Q\delta^{-1}$ by the malnormal property for peripheral 
subgroups mentioned above. This means that $P=Q$ and $\delta^{-1}\gamma\in P$ 
so $G_v$ and $G_w$ are in $\gamma P\gamma^{-1}$. 
We now continue until we find that
$G_v$ is in the same conjugate $\gamma P\gamma^{-1}$
of the same peripheral subgroup $P$ for all $v\in\Gamma$.

We now add the stable letters $t_i$: as each $t_i$ conjugates a non trivial
subgroup of $G_v$ to one of $G_w$ for some $v,w\in\Gamma$, we see that
$t_i(\gamma P\gamma^{-1})t_i^{-1}\cap 
\gamma P\gamma^{-1}$ is non trivial, so $t_i$ and hence the
whole group $G$ is in $\gamma P\gamma^{-1}$, thus $G=P$.
\end{proof}

We now need to give conditions on the vertex subgroups of a graph of groups
$G(\Gamma)$ in order to ensure that the conditions of Theorem
\ref{nonhyp} apply. By Lemma \ref{hypr} if each vertex group $G_v$ has an
element whose intersector is all of $G_v$ and $G_v$ is not infinite cyclic
then we can apply Theorem \ref{nonhyp} to conclude that $G$ is not relatively
hyperbolic. However we can repeat this idea by taking bigger and bigger
subgroups of $G_v$ with the same properties, thus giving rise to the
next definition.
\begin{defn}
If $H$ is any non trivial subgroup of a group $G$ then we define the
{\bf full intersecting conjugate} $F_G(H)=F^1_G(H)$ of 
$H=F^0_G(H)$ in $G$ to be the
subgroup
\[\langle g\in G:gHg^{-1}\cap H\neq\{e\}\rangle\]
of $G$ which contains $H$. We then inductively define 
\[F^{n+1}_G(H)=F_G(F^n_G(H))\mbox{ and }
Mal_G(H)=\cup_{n=1}^\infty F^n_G(H).\]
\end{defn}
(Note for an element $a\in G$ of
infinite order we have $F^1_G(\langle a\rangle)=F_G(\langle a\rangle)=
I_G(a)$.)
This is an ascending union of subgroups so also a subgroup and has the
following properties:
\begin{lem} \label{mlcl}
(i) If $S\leq H$ and $H\leq G$ then $Mal_G(S)\leq Mal_G(H)$ and 
$Mal_H(S)\leq Mal_G(S)$.\\
(ii) $Mal_G(H)$ is malnormal in $G$.\\
(iii) If $M$ is a malnormal subgroup of $G$ containing $H$ then
$M$ contains $Mal_G(H)$.
\end{lem}
\begin{proof}
Part (i) follows directly from the definition so
first say that we have $g\in Mal_G(H)$ such that 
$gMal_G(H)g^{-1}\cap Mal_G(H)\neq\{e\}$, so we have
$m_1,m_2\in Mal_G(H)\setminus\{e\}$ with $gm_1g^{-1}=m_2$. Then we have
$N\in\N$ with $m_1,m_2\in F^N_G(H)$, thus 
$gF^N_G(H)g^{-1}\cap F^N_G(H)\neq\{e\}$ and therefore 
$g\in F_G(F^N_G(H))\subseteq Mal_G(H)$.

Now suppose inductively that $F^N_G(H)\subseteq M$. If
$gF^N_G(H)g^{-1}\cap F^N_G(H)\neq\{e\}$ then $gMg^{-1}\cap M\neq\{e\}$ giving
$g\in M$, so all generators of $F^{N+1}_G(H)$ are in the subgroup
$M$ thus $F^{N+1}_G(H)$ is also and $Mal_G(H)$ is the union of these.  
\end{proof}
Thus one could perhaps call $Mal_G(H)$ the malnormal closure of $H$ in $G$.
\begin{co} \label{nonhypco}
Let $G(\Gamma)$ be a finite graph of groups where
each vertex group is non trivial and torsion free 
but with no further restriction on the edge groups other than they 
are all non trivial. Suppose that for all
vertices $v\in\Gamma$ we have $g_v\in G_v$ such that
$Mal_{G_v}(\langle g_v\rangle)=G_v$.
Then $G$ is not hyperbolic relative to any
collection of proper subgroups.
\end{co}
\begin{proof}
First suppose that no vertex groups are copies of $\Z$. Then
we can assume that $F^1_G(\langle g_v\rangle)$ is non cyclic, because
if we have an element $g$ in a group $G$ with 
$I_G(g)=F^1_G(\langle g\rangle)=\langle c\rangle$ then 
$F^2_G(\langle g\rangle)=F^1_G(\langle c\rangle)=I_G(c)$, whereas
$g=c^i$ means that $I_G(g)=I_G(c)$ and thus our ascending union
$Mal_G(\langle g\rangle)$
terminates in the cyclic group $\langle c\rangle$, but here
$Mal_G(\langle g_v\rangle)$ contains $Mal_{G_v}(\langle g_v\rangle)=G_v$
by Lemma \ref{mlcl} (i).
Thus by Lemma \ref{hypr} we
see that we have a peripheral subgroup $P$ and an element
$\gamma\in G$ such that 
$I_G(g_v)=F^1_G(\langle g_v\rangle)\subseteq\gamma P\gamma^{-1}$.
But by malnormality of $\gamma P\gamma^{-1}$ in $G$ we have by Lemma
\ref{mlcl} (iii) that $\gamma P\gamma^{-1}$ contains 
$Mal_G(\langle g_v\rangle)$ which itself contains 
$Mal_{G_v}(\langle g_v\rangle)=G_v$, 
so now we can apply Theorem \ref{nonhyp}.

Now say that some vertex groups are copies of $\Z$. We begin as before
by taking a vertex group $G_{v_0}\not\cong\Z$ (if none exist then we have a
generalised Baumslag-Solitar group as in \cite{krop} which has an
$s$-normal infinite cyclic subgroup and so cannot even be
acylindrically hyperbolic) and proceed similarly, so that 
$G_{v_0}$ is in 
$\gamma P\gamma^{-1}$ along with $G_v$ for all other vertices
encountered so far. If at any stage we now have $G_w=\langle z\rangle\cong\Z$,
where $w$ is adjacent to $v$ in $T$ and we already have 
$G_v \subseteq \gamma P\gamma^{-1}$ then we find that $z^i\in G_v$, where 
$\langle z^i\rangle$ is the edge group inclusion into $G_w$, so
$z^i\in \gamma P\gamma^{-1}$ but then 
$z^i\in z(\gamma P\gamma^{-1})z^{-1}\cap \gamma P\gamma^{-1}$ means that
$z$, and hence $G_w$, is in $\gamma P\gamma^{-1}$ too.
\end{proof}
\hfill\\
{\bf Examples}\\ 
If a torsion free group $S$ is soluble then let us
take an element $s$ in the final non trivial term $S^{(n)}$ of the derived
series. Clearly $s\in S^{(n)}\leq F_S^1(\langle s\rangle)$ because
$S^{(n)}$ is abelian, but then $S^{(i+1)}$ being normal in $S^{(i)}$
means that all $S^{(i)}$ and thus $S$ are in $Mal_S(\langle s\rangle)$
too. Now suppose we have a torsion free group $H$ with an
infinite soluble normal subgroup $S$ as above.
Then $S\leq F_H^{n+1}(\langle s\rangle)$ implies that
$H\leq F_H^{n+2}(\langle s\rangle)$. Let us further suppose that
the torsion free group $G$ has a finite index subgroup $H$ which possesses
an infinite soluble normal subgroup $S$ then, with $S$ and $s$ as before,
we obtain $H\leq F_G^{n+2}(\langle s\rangle)$ and so
$G\leq F_G^{n+3}(\langle s\rangle)\leq Mal_G(\langle s\rangle)$.
Thus we have
\begin{co}
Suppose that $G(\Gamma)$ is a finite graph of groups where each vertex 
group is
torsion free and contains a finite index subgroup which itself has an
infinite soluble normal subgroup, and
where the edge groups are all non trivial. Then $G$ is not hyperbolic
relative to any collection of proper subgroups.
\end{co}

As an example from 3-manifolds, a compact orientable irreducible 3-manifold
is a graph manifold if all components in its JSJ decomposition
are Siefert fibred spaces. In terms of the fundamental group, we can
describe this as a graph of groups where each edge group is non trivial
and where each vertex group has  a finite index subgroup which in turn
has an infinite cyclic normal subgroup. Thus we see that the fundamental
groups of graph manifolds are never relatively hyperbolic with respect
to any collection of proper subgroups, as opposed to when there are
hyperbolic pieces in the decomposition. In this case it is well known
that the fundamental group is hyperbolic relative to the maximal graph 
manifold pieces, including $\Z^2$ subgroups for tori bounding 
hyperbolic pieces on both sides. 

\section[Balanced HNN extensions and amalgams]{Balanced 
HNN extensions and amalgamated free products}
Although the property of being torsion free is preserved under HNN extensions
and amalgamated free products, being torsion free and balanced is not.
For HNN extensions this is obvious even with edge groups that are infinite
cyclic. For amalgamated free products, we have examples such as
$A$ is the free group on $a,b$ and $X$ is the free group on $x,y$ 
but we amalgamate
the rank 2 free subgroup $C=\langle a^3,bab^{-1}\rangle$ of $A$ with
the isomorphic subgroup $\langle x^2,yxy^{-1}\rangle$ of $X$ via
$a^3=x^2,bab^{-1}=yxy^{-1}$ to form $A*_CX$ in which 
$(y^{-1}b)a^2(y^{-1}b)^{-1}=a^3$ holds but $a$ is not trivial. However
in this section we will examine how intersectors change
when forming amalgamated free products or HNN extensions over an
infinite cyclic edge group, as this will provide necessary and sufficient
conditions as to when the property of being balanced is preserved under
these constructions.
These will then be applied to the more general
graph of groups construction in the next two sections. First we consider
amalgamated free products $G=A*_CB$. If $g\in G$ is not in $C$ then
we can express $g$ as
$g_r\ldots g_2g_1$ for length $r\geq 1$ 
and $g_1,g_2,\ldots ,g_r$ coming alternately from $A\setminus C$ and 
$B\setminus C$ (not uniquely though), 
which we will refer to as a reduced form for $g$. 
Conversely an element of this 
form cannot equal the identity or lie in $C$; indeed if $r\geq 2$
then it cannot even lie in $A\cup B$.
\begin{thm} \label{amal} If $A$ and $B$ are balanced torsion free groups and
$G=A*_CB$ for infinite cyclic $C$ then $G$ is also balanced.
\end{thm}
\begin{proof} Suppose otherwise, so
that there is $x\in G$ with a power $x^m$ conjugate in $G$ to
$x^n$ for $|m|\neq |n|$. On considering the action of $G$ on the
associated Bass-Serre tree, we see that all hyperbolic elements
are balanced (as they have non zero translation lengths), so $x$
must be conjugate in $G$ into either $A$ or $B$ but conjugates of
balanced elements are balanced. Thus without loss of generality
$x=a\in A$ and $ga^mg^{-1}=a^n$ for some $g\in G$, whereupon
we must have $g\notin A$ because $a$ is balanced in $A$ by
assumption.

We first suppose that no power of $a$ is conjugate in $A$ into $C$.
On expressing $g=g_r\ldots g_1$ in reduced form for $r\geq 1$, we have
that $a^m\in A\setminus C$, so if $g_1\in B\setminus C$ then $ga^mg^{-1}$
is in reduced form when written as $g_r\ldots g_1a^mg_1^{-1}\ldots g_r^{-1}$
and so is not even in $A$, let alone equal to a power of $a$. If however
$g_1\in A\setminus C$, whereupon we would have $r\geq 2$ as $g\notin A$,
then also 
$g_1a^mg_1^{-1}\in A\setminus C$ by our assumption, so now $g$ is in reduced
form of length at least three
when written as $g_r\ldots g_2(g_1a^mg_1^{-1})g_2^{-1}\ldots g_r^{-1}$
and hence is not in $A$. 

Now suppose that $C=\langle c\rangle$ and some power $a^i$ say of $a$
is conjugate in $A$ into $C$, so by replacing $a$ and $a^i$ with the relevant
conjugates in $A$ we can assume $a^i$ is equal to some power $c^j$ of $c$.
Establishing that $c$ is unimodular in $G$ also implies that $c^j$ and
hence $a^i$ and $a$ are all unimodular in $G$ too. We clearly
have $\langle I_A(c),I_B(c)\rangle\leq I_G(c)$ as $I_G(c)\cap H=I_H(c)$
for any subgroup $H$ of $G$ containing $c$, so we show containment
the other way: given $g\in G\setminus (A\cup B)$ with $g\in I_G(c)$,
so that $gc^kg^{-1}=c^l$ for some non zero $k,l$, we again write 
$g=g_r\ldots g_1$ in reduced form of length $r\geq 2$. But regardless
of whether $g_1\in A\setminus C$ or $B\setminus C$, we have two cases:
case 1 is that $g_1c^kg_1^{-1}$ is in $A\setminus C$ or $B\setminus C$ and thus
$gc^kg^{-1}$ is in reduced form of length at least 3 so is not in $A\cup B$.
The other case is when $g_1c^kg_1^{-1}$ is in $C$ but then $g_1$ is in
$I_A(c)$ or $I_B(c)$ and either way we would find that $g_1c^kg_1^{-1}=
c^{\pm k}$ as $A$ and $B$ are balanced. By continuing in this way with
$g_2,\ldots ,g_r$, we find that either $gc^kg^{-1}$
terminates in a reduced form and so $g\notin I_G(c)$, or 
$gc^kg^{-1}=c^{\pm k}$ so that $|k|=|l|$ and all of $g_1,\ldots ,g_r$ were
in $I_A(c)\cup I_B(c)$.
\end{proof}
\begin{co} \label{trco}
If $G$ is the fundamental group of a finite graph of torsion free, balanced
groups with all edge groups infinite cyclic and the graph is a tree then
$G$ is also balanced.  
\end{co}
\begin{proof} Build $G$ up by repeated amalgamations and use Theorem
\ref{amal} at each stage.
\end{proof}

Now we come to HNN extensions, whereupon it is clear that we can create
non balanced groups, for instance Baumslag-Solitar groups, from
balanced ones. We suppose that $G$ is an HNN extension of the base group
$H$ and associated isomorphic subgroups $A,B$ of $H$,
with $tAt^{-1}=B$. Again we need the concept of a reduced form in that 
any element $g\in G\setminus H$ can be expressed as
$h_rt^{\epsilon_r}\ldots h_1t^{\epsilon_1}h_0$ of length
$r\geq 1$,
$h_i\in H$, $\epsilon_i\in\{\pm 1\}$ and no pinch (an
appearance of $th_jt^{-1}$ for $h_j\in A$
or $t^{-1}h_jt$ for $h_j\in B$) occurs, and conversely an element in such 
form does not lie in $H$. This allows us to give sufficient conditions
under which the HNN extension is also balanced.
\begin{prop} \label{hnnp}
If $G$ is an HNN extension of the balanced, torsion free group $H$ with
stable letter $t$ and infinite cyclic associated subgroups 
$A=\langle a\rangle$ and $B=\langle b\rangle$ of $H$ so that $tat^{-1}=b$
then:\\
(i) If $h\in H$ but no power of $h$ is conjugate in $H$ into $A\cup B$
then $h$ is still unimodular in $G$.\\
(ii) If no conjugate in $H$ of $B$ intersects $A$ non trivially then $G$ is
also a balanced group.
\end{prop}
\begin{proof}
On being given $h\in H$,
suppose that there is $g\in G\setminus H$ with $gh^ig^{-1}=h^j$ and
$g=h_rt^{\epsilon_r}\ldots h_1t^{\epsilon_1}h_0$ in reduced form. Then
$gh^ig^{-1}$ is also in reduced form and hence not in $H$ unless
$\epsilon_1=+1$ and $h_0h^ih_0^{-1}\in A$ or $\epsilon_1=-1$ and
$h_0h^ih_0^{-1}\in B$. But neither of these occur in (i) so
$I_H(h)=I_G(h)$ and $h$ is also unimodular in $G$. 

For (ii)
we again note that by using conjugacy and the Bass-Serre tree, we need
only check the unimodularity of elements in $H$, so by (i) we can now
assume that without loss of generality some power of $h$ lies in $A$
but that no power of $h$ is conjugate in $H$ into $B$. Thus if $gh^ig^{-1}$
is to be an element of $H$ in the
above then we can only have $\epsilon_1=+1$ and 
$h_0h^ih_0^{-1}\in A$, say $a^k$ so that 
\[gh^ig^{-1}=h_rt^{\epsilon_r}\ldots t^{\epsilon_2}h_1b^kh_1^{-1}
t^{-\epsilon_2}\ldots t^{-\epsilon_r}h_r^{-1}\] because $tat^{-1}=b$. 
But $h_1b^kh_1^{-1}$
cannot be in $A$ so again we are reduced if $\epsilon_2=+1$, or if
$\epsilon_2=-1$ and $h_1b^kh_1^{-1}$ is not in $B$. But if
$h_1b^kh_1^{-1}=b^l$ then $|k|=|l|$ because $b$ is unimodular in $H$.
By continuing in this way, we see that either $gh^ig^{-1}$
terminates in a reduced word not lying in $H$, or we merely pass through
$b^{\pm k}$ or $a^{\pm k}$ as we evaluate $gh^ig^{-1}$
from the middle outwards. But if we end up at the last step with 
$h^j=gh^ig^{-1}=h_rb^{\pm k}h_r^{-1}$ then $h^j$ would be conjugate in $H$
into $B$. Thus here we can only end up with 
\[h^j=gh^ig^{-1}=h_ra^{\pm k}h_r^{-1}=h_rh_0h^{\pm i}h_0^{-1}h_r^{-1}\]
so $|i|=|j|$ as $h$ is unimodular in $H$.
\end{proof}

We can now give the exact condition on when such an HNN extension would
not be balanced.
\begin{thm} \label{hnnt}
Let $G$ be an HNN extension of the balanced, torsion free group $H$ with
stable letter $t$ and infinite cyclic associated subgroups 
$A=\langle a\rangle$ and $B=\langle b\rangle$ of $H$ where $tat^{-1}=b$.
Suppose further that there is $h\in H$ conjugating a power of $a$ to a
power of $b$, so that $ha^ih^{-1}=b^j$. 
Then $G$ is balanced if and only if $|i|=|j|$.
\end{thm}
\begin{proof}
We first replace $t$ with the alternative stable letter $s=h^{-1}t$ which
conjugates $A$ to $h^{-1}Bh$, so that $sa^js^{-1}=a^i$ and hence $|i|=|j|$
is a necessary condition for $G$ to be balanced, so we assume this for now
on. By Proposition \ref{hnnp} the only elements we need to show are balanced
are those in $H$ having a power conjugate in $H$ into $A\cup h^{-1}Bh$, so
without loss of generality we take $h\in H$ with a power conjugate in $H$
into $A$, but then we need only show that this power of $a$ is unimodular
and this reduces to looking at $I_G(a)$ which we will now show is equal to
$\langle s,I_H(a)\rangle$.

As $s\in I_G(a)$ already, we again take
$g=h_rs^{\epsilon_r}\ldots s^{\epsilon_1}h_0\in G\setminus H$ 
and in reduced form, along with an arbitrary power $a^k$ of $a$. As before
$ga^kg^{-1}$ will be reduced unless at least $\epsilon_1=+1$ and 
$h_0a^kh_0^{-1}\in A$, in which case it is equal to $a^{\pm k}$ as $a$ is
unimodular in $H$, or $\epsilon_1=-1$ and $h_0a^kh_0^{-1}\in h^{-1}Bh$. But
the former case means that
\[s^{\epsilon_1}h_0a^kh_0^{-1}s^{-\epsilon_1}=h^{-1}b^{\pm k}h\]
whereas if we had $h_0a^kh_0^{-1}=h^{-1}b^lh$ in
the latter case then $h_0a^{ik}h_0^{-1}=h^{-1}b^{il}h=a^{\pm il}$, so also
$|k|=|l|$ and therefore $s^{\epsilon_1}h_0a^kh_0^{-1}s^{-\epsilon_1}
=t^{-1}b^{\pm k}t=a^{\pm k}$. Thus again we see that we move between
$a^{\pm k}$ and $h^{-1}b^{\pm k}h$ as we conjugate, but if the latter
is a power of $a$ then it can only be $a^{\pm k}$ as above.
Hence if $g\in I_G(a)$
then $ga^kg^{-1}$ can only equal $a^{\pm k}$.
\end{proof}
We can now apply the above results to graphs of groups, but in order to
use these repeatedly we need to know how conjugacy works in these cases.

\section[Conjugacy in graphs of groups]{Conjugacy in graphs of balanced groups 
with infinite cyclic edge groups}
We assume the usual definition and standard facts about a finite graph
of groups $G(\Gamma)$ with fundamental group $G$ where the underlying graph
$\Gamma$ has vertices $V(\Gamma)$ and (unoriented) edges $E(\Gamma)$.
We write $G_v$ for the vertex group at $v\in V(\Gamma)$, whereas on taking
an edge $e\in E(\Gamma)$ and giving it an orientation so that it travels
from the vertex $v_1$ to $v_2$ (where possibly $v_1=v_2$), we write
$G_{e^-}$ for the inclusion of the edge group in $G_{v_1}$ and $G_{e^+}$
for the inclusion into $G_{v_2}$.
We also assume here that all vertex groups are torsion free and all edge
groups are infinite cyclic. We define a vertex element $g_v\in G_v$ of
$G(\Gamma)$ to be an element that actually lies in some particular
vertex group, not 
just one that is conjugate into it. (Strictly speaking this is not well
defined but it is once a maximal subtree of $\Gamma$ is specified.) It seems
we need to know when two vertex elements are
conjugate in $G$; in fact it turns out we only ever need to know
when they have powers that are conjugate
in $G$. The most obvious way in which this could hold is if they lie
in the same vertex group and have powers that are conjugate in this
vertex group. Our next definition incorporates what we regard as the
second most obvious way.
\begin{defn} \label{path}
Given a graph of groups $G(\Gamma)$ with torsion free vertex groups and
infinite cyclic edge groups, along with vertex elements $g\in G_v$ and
$g'\in G_{v'}$, we define a {\bf conjugacy path} $p$ from $g$ to $g'$ to be an
oriented non empty
edge path $v=v_0,v_1,\ldots ,v_n=v'\in V(\Gamma)$ traversed by edges
$e_1,\ldots ,e_n\in E(\Gamma)$ for $n\geq 1$ with $e_i$
given the orientation from $v_{i-1}$ to $v_i$ such that the following
conditions hold:\\
(1) Some power of $g$ is conjugate in $G_{v_0}$ into the edge subgroup
$G_{e_1^-}$ of $G_{v_0}$.\\
(2) For each $i=1,\ldots ,n-1$, some conjugate in $G_{v_i}$ of the edge
subgroup $G_{{e_i}^+}$ intersects the edge subgroup $G_{e_{i+1}^-}$ non
trivially.\\
(3) Some power of $g'$ is conjugate in $G_{v_n}$ into the edge subgroup
$G_{e_n^-}$ of $G_{v_n}$.\\
We say that the conjugacy path $p$ is {\bf reduced} and/or {\bf closed} if the
underlying edge path is reduced and/or closed.
\end{defn}
Note: Every conjugacy path gives rise to a unique non empty edge path, but
conversely suppose we have a non empty edge path from $v$ to $v'$ 
and elements $g,g'$ in $G_v$ and $G_{v'}$ respectively.
Then we say this edge path {\bf induces} a conjugacy path from $g$ to
$g'$ if (1), (2) and (3) all hold.
 
If a conjugacy path exists from $g$ to $g'$
then it is clear that some power of $g$ and some
power of $g'$ are conjugate in $G$: certainly a power of $g$ is conjugate
into $G_{e_1}^-$ which is either equal to or conjugate in $G$ to $G_{e_1^+}$,
depending on whether we form an amalgamated free product or HNN extension
over the edge $e_1$. But some element of $G_{e_1^+}$ is conjugate in $G$
into
$G_{e_2^-}$, and by taking higher powers if necessary we can assume that
this element is conjugate in $G$ to a power of $g$. 
We then continue in this way
until we reach an element of $G_{e_n^+}$, and some power of this will be
conjugate to a power of $g'$.

Our next step is to show that in order to establish the existence of
conjugacy paths between two vertex elements, we need only use
reduced paths. This is certainly not the case if the edge groups are not
infinite cyclic: for instance consider the amalgamated free product
$G=A*_CX$ where $A$ is free on $a,b$ and $X$ is free on $x,y$,
with $C$ also a rank 2 free group and the amalgamation defined by
identifying $a^2$ with $x$ and $b^2$ with $yxy^{-1}$. Then $a^2$ and
$b^2$ are conjugate in $G$ but not in $A$, though any conjugacy path
establishing this will need to leave $A$ and then return, 
thus will not be reduced. 
In fact there is some literature on conjugacy in graphs of groups, even
specialising in the case where edge groups are infinite cyclic. However
this is usually geared towards the conjugacy problem or conjugacy
separability. We have not seen the following results elsewhere, perhaps
because they are only concerned with conjugacy of unspecified powers, rather
than the elements themselves. Furthermore it will be shown not only that
we need just consider reduced paths but that it is enough to consider paths
that never pass through the same edge twice. This is important for
applications because it means that there are only ever finitely many
such paths to check. 
\begin{prop} \label{red}
If there exists a conjugacy path from $g\in G_{v_0}$ to $g'\in G_{v_n}$ then
either $v_0=v_n$ and $g,g'$ have powers which are conjugate in this vertex 
group, or
we can take the underlying edge path to be reduced: indeed we can assume that 
this edge path only traverses any unoriented edge at most once.
\end{prop}
\begin{proof}
With the notation in Definition \ref{path}, suppose that
$e_r$ and $e_s$ (for $r<s$) in the underlying edge path $p$ 
are the same edge, running
from $v_{r-1}$ to $v_r$. Then some conjugate in $G_{v_r}$ of the edge 
subgroup $G_{e_r^+}=G_{e_s^+}$ intersects both edge subgroups
$G_{e_{r+1}^-}$ and $G_{e_{s+1}^-}$ non trivially. In particular some
conjugate in $G_{v_r}$ of $G_{e_r^+}$ intersects $G_{e_{s+1}^-}$. Thus we can
remove from $p$ the edges $e_{r+1},\ldots ,e_s$ which run from $v_r$ back
to itself and we still have a conjugacy path from $g$ to $g'$. We now
continue until we have removed all such repeats.

If though $e_s$ is the reverse of $e_r$, running backwards
from $v_r$ to $v_{r-1}$ then the argument is similar
but now we have $G_{e_r^-}$ equal to $G_{e_s^+}$ in $G_{v_{r-1}}$ so
that there are conjugates in $G_{v_{r-1}}$ of $G_{e_{r-1}^+}$ and of
$G_{e_{s+1}^-}$ which both intersect $G_{e_r^-}=G_{e_s^+}$.
This time we can cut out the edges $e_r,\ldots ,e_s$ and still have a
conjugacy path from $g$ to $g'$, unless we have cut out the whole path.
In this case the edge $e_r$ is actually $e_1$ from $v_0$ to $v_1$ and
the edge $e_s=e_n$ is the reverse of $e_1$,
with a power of $g$ (respectively $g'$)
conjugate in $G_{v_0}$ (respectively $G_{v_n}$ which is equal to $G_{v_0}$)
into $G_{e_1^-}$ (respectively $G_{e_n^+}$ which is equal to $G_{e_1^-}$).
Thus there are powers of $g$ and $g'$ which
are already conjugate in $G_{v_0}$.  
\end{proof}
 
We can now consider when there exist powers of two vertex elements which are
conjugate in the fundamental group of a graph of torsion free groups
with infinite cyclic edge groups.
We start when the graph is a tree, though we first note the conjugacy
theorem for amalgamated free products in \cite{mks}. Here we only require a
partial version which can easily be proven by using reduced forms as in
the previous section.
\begin{prop} \label{cnjam}
Let $G=A*_CB$ be an amalgamated free product and let $g\in A$.\\
If $g'\in A\cup B$ and $g,g'$ are conjugate in $G$ but $g$ is
not conjugate in $A$ into $C$ then $g'\in A$
with $g,g'$ conjugate in $A$.
\end{prop}
\begin{thm} \label{cnjtr}
Let $G(\Gamma)$ be a finite graph of torsion free groups with
infinite cyclic edge groups and where $\Gamma$ is a tree.  
Suppose we have two vertex elements $g\in G_v$ and $g'\in G_{v'}$. Then
some power of $g$ is conjugate in $G$ to some power of $g'$ 
if and only if either\\
(i) The vertices $v,v'$ are equal and some power of $g$ is conjugate in 
$G_v$ to some power of $g'$.\\
(ii) The vertices $v,v'$ are distinct and the unique reduced path in $\Gamma$
from $v$ to $v'$ induces a conjugacy path from $g$ to $g'$.
\end{thm}
\begin{proof}
We have seen that
these conditions are sufficient for conjugacy so we prove necessity by
induction on the number of edges in $\Gamma$, with Proposition \ref{cnjam}
being our base case: if $G=A*_CB$ has one edge and $g,g'$ are both in $A$
with a power of $g$ conjugate in $G$ to a power of $g'$ then either these
powers are conjugate in $A$, or they are both conjugate in $A$ into $C$ and so
further powers of $g$ and $g'$ are conjugate in $A$ to each other anyway. 
Alternatively if $g\in A$ but $g'\in B$ and $g^i$ is conjugate in $G$ to $g'^j$
then $g^i$ is conjugate in $A$ into $C$ and $g'^j$ conjugate in $B$ into $C$,
giving us our conjugacy path from $g$ to $g'$.

In the general case we again suppose that 
some power of $g$ is conjugate in $G$ to some power of $g'$. First suppose that
$v=v'$. We remove a leaf vertex $w\neq v$ in
$\Gamma$ and its adjoining edge $e_w$ from $\Gamma$ to form the tree
$\Gamma_w$. Now we can form the fundamental group $A$ of the graph of
groups given by the tree $\Gamma_w$, which means that $G=A*_CG_w$, where
$C$ is the infinite cyclic edge group $G_{e_w}$. Now as $v\in\Gamma_w$
we have that $g,g'\in A$. But, as in the base case,
if both a power of $g$ and a power of $g'$
are conjugate in $A$ into $C$ then some power of $g$ is conjugate in $A$
to some power of $g'$. If not then $g$ say has no power conjugate in $A$
into $C$, in which case we have by Proposition \ref{cnjam} that any
element of $A$ which is
conjugate in $G=A*_CB$ into $\langle g\rangle$ must already be
conjugate in $A$. Either way we are now conjugate back in $A$ so the induction
holds to obtain conjugates of a power of $g$ and a power of $g'$ back in
$G_v$.

To prove (ii), first suppose that no power of $g$ is conjugate in $G_v$
into the edge inclusion $G_{{e_1}^-}$. We similarly remove the edge $e_1$ from
$\Gamma$ to form trees $\Gamma_0,\Gamma_1$ and the amalgamated free
product $A*_CB$, where $A$ is the fundamental group of the graph of groups
obtained from $\Gamma_0$, $B$ from $\Gamma_1$ and $C=G_{e_1^-}$.
Hence all powers of
$g$ are in $A$ and all powers of $g'$ are in $B$, but we can only have the
element $g^i$ of $A$ conjugate to an element of $B$
if $g^i$ is conjugate in $A$ into $C=\langle c\rangle$. As the graph
$\Gamma_0$ is a tree, we can use induction to determine if a power of
$g\in G_v$ is conjugate in $A$ to a power of $c\in G_v$, whereupon we
see this occurs if and only if these powers were already conjugate in $G_v$.

 We then argue in the same way if no power of $g'$ is conjugate in $G_{v'}$
into the subgroup $G_{e_n^+}$ of $G_{v'}$. Otherwise the only way that
(ii) can fail is if somewhere along this edge path joining $v$ to $v'$,
say at the vertex $v_i$, we have that no conjugate of $G_{e_i^+}$ in
$G_{v_i}$ meets $G_{e_{i+1}^-}$, and we suppose that this is the first
time it occurs. We remove the edge $e_{i+1}$ giving us the decomposition
$G=A*_CB$ where $C$ embeds as $G_{e_{i+1}^-}$ in $A$ and as $G_{e_{i+1}^+}$
in $B$. Now some power of $g$ is this time conjugate in $G_{v_0}$ (for
$v_0=v$) into $G_{e_1^-}$, and hence by following the path from $v_0$ up to
$v_i$, we have that a power of $g$ is conjugate in $A$ into the edge
inclusion $G_{e_i^+}$. Now suppose that some other power of $g$ is conjugate
in $G$ into $\langle g'\rangle\leq B$. By Proposition \ref{cnjam} this
can only happen if this new power of $g$ is conjugate in $A$ into
$C=G_{e_{i+1}^-}$. But this would force a further power of $g$ to be conjugate
in $A$ both into $G_{e_i^+}$ and into $G_{e_{i+1}^-}$, hence $G_{e_i^+}$
can be conjugated in $A$ to meet $G_{e_{i+1}^-}$, and hence by induction
also in $G_{v_i}$ which is a contradiction.  
\end{proof}

We now need to provide a similar result for general graphs of groups. Thus
we require the equivalent version of Proposition \ref{cnjam} for HNN
extensions, usually known as Collins' criterion. Again we just require
the simplified version below which can be verified with reduced forms.
\begin{prop} \label{cnjhnn}
Suppose that $G$ is an 
HNN extension of $H$ with stable letter $t$ and associated
subgroups $A,B$ of $H$ so that $tAt^{-1}=B$ and let $g\in H$ be conjugate
in $G$ to $g'\in H$. Then\\
If $g$ is not conjugate in $H$ into $A\cup B$ then nor is $g'$ and
$g$ is conjugate in $H$ to $g'$.
\end{prop}
\begin{thm} \label{cnjgr}
Let $G(\Gamma)$ be a finite graph of torsion free groups with
infinite cyclic edge groups.  
Suppose we have two vertex elements $g\in G_v$ and $g'\in G_{v'}$. Then
some power of $g$ is conjugate in $G$ to some power of $g'$ 
if and only if either:\\
(i) The vertices $v,v'$ are equal and some power of $g$ is conjugate in 
$G_v$ to some power of $g'$.\\
(ii) There exists some conjugacy path from $g$ to
$g'$ (where closed paths with $v=v'$ are allowed). 
\end{thm}
\begin{proof}
Again these conditions clearly imply conjugacy so we suppose that a power
of $g$ and a power of $g'$ are conjugate in $G$.
We have our result when $\Gamma$ is a tree, so we now take a maximal tree $T$
in $\Gamma$ and argue by induction on the number of remaining edges, with the
base case being Theorem \ref{cnjtr}.

Suppose on removing from $\Gamma$ an edge $e$
not in $T$ we are left with the
connected graph $\Delta$, giving rise to the graph of groups $H(\Delta)$
so that $G$ is formed from $H$ by an HNN extension with associated subgroups
$A=G_{e^-}\leq G_w\leq H$ and 
$G_{e^+}\leq G_{w'}\leq H$, where $e$ runs from the vertex
$w$ to $w'$ (where we could have $w=w'$).
Again we ask: is this power of $g$ conjugate in $H$ into $A\cup B$? If not
then our powers of $g$ and $g'$ must be conjugate in $H$ by Proposition
\ref{cnjhnn}, and so we can inductively use the criterion of conjugacy in
$H(\Delta)$ instead. Otherwise this power is conjugate in $H$ to
some element $a$ say of $A$ without loss of generality, and we must also
have in this case that our power of $g'$ is conjugate in $H$ into either
$A$ or into $B$ by Proposition \ref{cnjhnn} with $g$ and $g'$ swapped.
If it is $A=G_{e^-}$ then as before some other power of
$g$ will be conjugate in $H$ to some power of $g'$ so that we are back
with the inductive statement for $H(\Delta)$. 

If however this power of
$g'$ is conjugate in $H$ to $b\in B$ say then we can at least
use the induction to say that there is a conjugacy path in $H(\Delta)$,
and hence in $G(\Gamma)$,
from $g$ to $a$ which joins the vertices $v$ and $w$, as well as a
conjugacy path from $g'$ to $b$ joining $v'$ and $w'$ (or we have conjugacy
of powers within the relevant vertex groups).
But as the edge
$e$ induces a conjugacy path in $G(\Gamma)$ from $a$ to $b$, we can put these
together to get one from $g$ to $g'$.
\end{proof}

\section{Graphs of balanced groups}
We can now put together the results of the previous two sections.
\begin{defn} Suppose that $G(\Gamma)$ is
a finite graph of balanced, torsion free groups with
infinite cyclic edge groups.
We say a reduced, closed, conjugacy path $e_1,\ldots ,e_n$ 
from $g\in G_v$ to $g'$ in the
same vertex group $G_v$, where $g$ generates the edge group inclusion
$G_{e_1^-}\leq G_v$ and $g'$ generates $G_{e_n^+}\leq G_v$   
is {\bf complete} if some powers of $g$ and $g'$
are themselves conjugate in $G_v$.
If so then, as the existence of the conjugacy path implies that
some power of $g$ is conjugate in $G$ to some power of $g'$, we obtain
two powers $g^i,g^j$ of $g$ which are conjugate in $G$.
We say that our complete conjugacy path is {\bf level}
if $|i|=|j|$. This is well defined as all vertex groups are balanced and it
does not depend on which power $g^i$ we initially
take provided it can be conjugated all the way round the conjugacy path.
\end{defn}
\begin{prop} \label{redcmp}
Given a finite graph $G(\Gamma)$ of balanced, torsion free groups with
infinite cyclic edge groups and a complete conjugacy path from $g\in G_v$
to $g'\in G_v$, there exists a complete conjugacy path which passes through
every non oriented edge of $\Gamma$ at most once, though possibly
starting (and thus ending) at a different vertex. If our original conjugacy
path is not level then we can arrange that this new path is not level either.
\end{prop}
\begin{proof}
Apply the method in the proof of Proposition \ref{red}, so that if
$e_r$ and $e_s$ (for $r<s$) are the same oriented
edge then we create two conjugacy
paths $e_1,\ldots ,e_r,e_{s+1},\ldots ,e_n$ and $e_{r+1},\ldots ,e_s$ which
are both closed, reduced and complete. If both paths are level then on putting
these together, so was the original path.

The same argument works if $e_s$ is the reverse of $e_r$ for $r$ minimal
with this property and this will create the two new paths
$e_1,\ldots ,e_{r-1},e_{s+1},\ldots ,e_n$ and $e_{r+1},\ldots ,e_{s-1}$,
except that the first edge path might be empty if $r=1$ and $s=n$ (though
the second will never be as complete conjugacy paths are reduced).
But if so then the fact that the original path is
not level implies that the second path is not either.
\end{proof}
\begin{thm} \label{bal}
Given a finite graph $G(\Gamma)$ of balanced, torsion free groups with
infinite cyclic edge groups, we have that the fundamental group $G$
is not balanced if and only if there exists a complete non level
conjugacy path in $G(\Gamma)$. By Proposition \ref{redcmp} 
this path can be taken to pass through
any unoriented edge at most once and so, as the underlying edge path will
determine a complete conjugacy path including the initial and final
elements (or their inverses), there are only finitely many paths to check.
\end{thm}
\begin{proof}
Clearly a complete non level conjugacy path gives rise to an unbalanced
element, 
thus let us assume all complete conjugacy paths are level
and again take a maximal tree $T$ of $\Gamma$. On forming the graph
of groups $H(T)$ we have that this fundamental group $H$ is balanced
by Corollary \ref{trco}.   
Now suppose that $\Gamma\setminus T$ consists of the edges
$e_1,\ldots ,e_n$. On inserting the edge $e_1$ to form the graph of
groups $G_1(\Gamma_1)$, we have that $G_1$ is an HNN extension
$\langle H,t_1\rangle$. Let us suppose that $e_1$ runs from the vertex
$v_1$ to $w_1$, so that the associated subgroups of this HNN extension
are $\langle a\rangle=G_{e_1^-}$ in $G_{v_1}$ and 
$\langle b\rangle=G_{e_1^+}$ in $G_{w_1}$ with $t_1at_1^{-1}=b$. By
Proposition \ref{hnnp} we have that $G_1$ is still balanced if no conjugate
in $H$ of $\langle b\rangle$ intersects $\langle a\rangle$ non trivially, so
now suppose that there is $h\in H$ so that $ha^ih^{-1}=b^j$. By Theorem
\ref{cnjtr} applied to $H(T)$, this can only happen if our relation is
obtained from a reduced
conjugacy path in $H(T)$ from $a$ to $b$ that joins $v_1$
to $w_1$, or if $v_1=w_1$ and there is $\gamma\in G_v$ with 
$\gamma a^k\gamma^{-1}=b^l$. But then on adding the edge $e_1$ with stable
letter $t_1$, we have a conjugacy path from $a$ back to itself
where $h^{-1}t_1$ (or $\gamma^{-1}t_1$) conjugates $a^j$ to $a^i$ 
(or $a^l$ to $a^k$) and which is reduced and therefore complete.
Thus if $|j|\neq |i|$ (or $|l|\neq |k|$)
then this conjugacy path is not level and $G_1$, which will
be a subgroup of $G$, is not balanced. If however $|i|=|j|$ (or $|k|=|l|$)
then the path is level and $G_1$ is balanced by Theorem \ref{hnnt}.

We can then form further HNN extensions $G_2,G_3,\ldots $ by adding an
edge each time until we reach the fundamental group $G$ and the above
argument applies at each stage, but in place of
Theorem \ref{cnjtr} we use Theorem \ref{cnjgr} as well as Proposition \ref{red}
to ensure any conjugacy path used is reduced.
\end{proof}
In particular, for any graph of groups as in this theorem which has a complete
non level conjugacy path, the resulting fundamental group cannot be
subgroup separable and cannot be linear over $\Z$.
Moreover this fundamental group cannot lie in any of the classes of
balanced groups given in Section 2, even if all the vertex groups lie in
such a class.
In \cite{wsox} Theorem 5.1 it was shown using geometric means
that a finite graph of groups with all vertex
groups free and all edge groups infinite cyclic is subgroup separable
if and only if it is balanced. It is also mentioned that, as a consequence
of this proof, the unbalanced case can be seen from the graph of groups
and the corresponding description fits with Theorem \ref{bal}.

\section{Presence of Baumslag Solitar groups}
Although we saw in Section 3 
that an unbalanced torsion free group need not contain a
Baumslag Solitar group, we would now like
conditions on the vertex groups of a graph of torsion free groups which
ensure that if the fundamental group is unbalanced then it contains a
non Euclidean Baumslag Solitar subgroup. Although we do not show this
for all torsion free groups, we would at least like to ensure that it
holds when the vertex groups are all word hyperbolic or all free abelian,
or even any combination of these. The following condition allows a
considerable generalisation of this case.
\begin{defn} We say a torsion free group $G$ has the {\bf cohomological
condition} if whenever we have $a,b$ non trivial elements of $G$ such that
$\langle a\rangle\cap\langle b\rangle$ is non trivial, the subgroup
$\langle a,b\rangle$ of $G$ has cohomological dimension at most 2.
\end{defn}
\begin{prop} The following groups have the cohomological condition.\\
(0) Subgroups of groups with the cohomological condition\\
(i) Groups of cohomological dimension at most 2\\
(ii) Torsion free word hyperbolic groups\\
(iii) More generally, torsion free groups which are CT\\
(iv) Torsion free abelian groups, or more generally torsion free
nilpotent groups\\
(v) More generally still, groups that are residually (torsion free nilpotent)\\
(vi) Torsion free groups which are 
relatively hyperbolic with respect to subgroups that satisfy the 
cohomological condition.
\end{prop}
\begin{proof}
(0) is obvious, as is (i) by Shapiro's Lemma that a subgroup cannot increase
cohomological dimension. For (iii) if we have $a,b\in G$ with powers $a^r=b^s$
then $a$ and $b$ commute with this element, thus with each other and so
$\langle a,b\rangle$ is torsion free abelian, thus is $\Z$.
For (iv) we recall that a finitely generated torsion free nilpotent group
is either trivial, $\Z$ or it surjects to $\Z^2$ (for instance it is residually
finite-$p$ for every prime $p$ by a result of Gruenberg, thus it must surject
$(C_p)^2$ or be cyclic) and being torsion free nilpotent is preserved on
passing to subgroups. But if $a,b\in G$ with powers $a^r$ equal to $b^s$
then $\langle a,b\rangle$ cannot surject to $\Z^2$. For (v), being 
residually (torsion free nilpotent) is also preserved by subgroups, so
if $G$ is residually (torsion free nilpotent), thus torsion free, and we
have elements $a,b$
in $G$ with $a^r=b^s$ but $\langle a,b\rangle$ is not abelian
(otherwise it is $\Z$) then $\langle a,b\rangle$ surjects to a non abelian
torsion free nilpotent group. This must surject to $\Z^2$ thus so does
$\langle a,b\rangle$.

Finally if $G$ is relatively hyperbolic with respect to subgroups having
this condition and we have elements
$a,b$ of $G$ with $a^r=b^s$ then it follows
immediately from Lemma \ref{hypr} on taking
$g=a^r=b^s$ that either $\langle a,b\rangle$ is infinite cyclic or
it can be conjugated into a peripheral subgroup. 
\end{proof}
Groups of cohomological dimension 2 need not be balanced, thus groups
which are relatively hyperbolic with respect to these need not be either, 
but we saw earlier
that groups in all the other categories will be, apart from torsion free
CT groups where the exceptions were identified in Proposition \ref{ctid}.
 
We will want to show that if all vertex groups have this property and
are balanced then an unbalanced graph of groups with infinite cyclic
edge groups contains a Baumslag Solitar subgroup that is non Euclidean. 
We start by mentioning a couple of well known lemmas: the first following
from Mayer-Vietoris considerations and the second from the usual use
of reduced forms.
\begin{lem} \label{co2}
If $G(\Gamma)$ is a finite graph of groups with all vertex groups having
cohomological dimension at most 2 and all edge groups are infinite
cyclic then $G$ has cohomological dimension at most 2.
\end{lem}
\begin{lem} \label{subha}
(i) If $G=A*_CB$ is an amalgamated free product and we have subgroups
$A'\leq A$ and $B'\leq B$ which both contain $C$ then the subgroup
$\langle A',B'\rangle$ of $G$ can be expressed naturally as the amalgamated
free product $A'*_CB'$.\\
(ii) Suppose that $G=\langle H,t\rangle$ is an HNN extension with base $H$, 
stable letter $t$ and associated subgroups $A,B$ such that $tAt^{-1}=B$.\\ 
If we have a subgroup $R$ of $H$ which contains both $A$ and $B$ then
the subgroup $\langle R,t\rangle$ of $G$ is naturally
the HNN extension with base $R$, stable letter $t$ and associated subgroups
$A,B$ such that $tAt^{-1}=B$.\\
If we have subgroups $J,L$ of $H$ with $J$ containing $A$ and $L$ containing
$B$ then the subgroup $\langle tJt^{-1},L\rangle$ of $G$ can be
naturally expressed as the amalgamated free product 
$tJt^{-1}*_{tAt^{-1}=B}L$.
\end{lem}
We now transfer the cohomological property from the vertex groups
to the fundamental group of the graph of groups.
\begin{thm} \label{precor}
Let $G(\Gamma)$ be a finite graph of groups with all vertex groups satisfying
the cohomological condition and all edge groups infinite cyclic. Suppose
we have a conjugacy path between vertex elements $a,b$ of $G$, thus
providing us with an element $g\in G$ and $i,j\neq 0$ such that 
$ga^ig^{-1}=b^j$. Then 
$\langle gag^{-1},b\rangle$ has cohomological dimension at most 2.
\end{thm}
\begin{proof}
In the proof that follows, cohomological dimension 2 will actually stand
for cohomological dimension at most 2.
We first assume that $\Gamma$ is a tree $T$, so that we are in the same
set up as Theorem \ref{cnjtr} Case (ii). Then we have our edge path
$e_1,\ldots ,e_n$ running from $A=G_{v_0}$ containing $a$ to
$B=G_{v_n}$ containing $b$, with the edge $e_k$ running from the vertex
$v_{k-1}$ with vertex group $G_{v_{k-1}}$ to $v_k$ with vertex group
$G_{v_k}$. We set $\langle f_k\rangle$ equal to the edge group
$G_{e_k}$ and use this notation for its image in both of the
neighbouring vertex groups.

Writing out in order the conjugation equalities that hold
in each vertex group, we obtain
\[g_0a^{i_1}g_0^{-1}=f_1^{j_1},g_1f_1^{i_2}g_1^{-1}=f_2^{j_2},
\ldots,g_{n-1}f_{n-1}^{i_n}g_{n-1}^{-1}=f_n^{j_n},g_nf_n^{i_{n+1}}g_n^{-1}=
b^{j_{n+1}}\]
where $g_k$ is the conjugating element in the vertex group $G_{v_k}$
so that $g=g_ng_{n-1}\ldots g_1g_0$,
$f_k\in G_{v_{k-1}}\cap G_{v_k}$ and $i_1,\ldots ,i_{n+1}$ and
$j_1,\ldots ,j_{n+1}$ are integers such that $i_1i_2\ldots i_{n+1}=i$
and $j_1j_2\ldots j_{n+1}=j$.
Let us set $A_0$ to be the subgroup 
$\langle g_0ag_0^{-1},f_1\rangle$ of
$G_{v_0}$ and $B_0$ to be the subgroup $\langle f_1,g_1^{-1}f_2g_1\rangle$
of $G_{v_1}$. By the cohomological condition on the vertex subgroups and
the conjugation equalities above, both $A_0$ and $B_0$ have cohomological
dimension 2, thus $\langle g_0ag_0^{-1},f_1\rangle *_{\langle f_1\rangle}
\langle f_1,g_1^{-1}f_2g_1\rangle$ does too. But by Lemma \ref{subha} (i)
this is equal to the subgroup 
$H_1=\langle g_0ag_0^{-1},f_1,g_1^{-1}f_2g_1\rangle$
of $S_1=G_{v_0}*_{\langle f_1\rangle} G_{v_1}$, so the subgroup
$\langle g_1g_0a(g_1g_0)^{-1},f_2\rangle$ of $g_1H_1g_1^{-1}$ also has
cohomological dimension 2. We can now amalgamate this with the cohomological
dimension 2 subgroup $\langle f_2,g_2^{-1}f_3g_2\rangle$ of $G_{v_2}$ over the
subgroup $\langle f_2\rangle$ which results in a cohomological dimension
2 subgroup of 
$S_2=g_1S_1g_1^{-1}*_{\langle f_2\rangle}G_{v_2}$ and so on, until we
conclude that $\langle g_{n-1}\ldots g_1g_0a(g_{n-1}\ldots g_1g_0)^{-1},
g_n^{-1}bg_n\rangle$ has cohomological dimension 2, thus so does the
conjugate subgroup $\langle gag^{-1},b\rangle$. 

Now we move to the case where the graph $\Gamma$ is not a tree, so we
take a maximal tree $T$ in $\Gamma$ with fundamental group $H$.
Again we have a conjugacy path
from $a$ to $b$ and we are in the
same situation as for the tree except that this time we might
walk over edges $e_k$ that are not in $T$, in which case our
conjugating equality $g_{k-1}f_{k-1}^{i_k}g_{k-1}^{-1}=
f_k^{j_k}$ is still the same, but the following
conjugating equality which was previously
$g_kf_k^{i_{k+1}}g_k^{-1}=
f_{k+1}^{j_{k+1}}$ is now
$g_ktf_k^{i_{k+1}}t^{-1}g_k^{-1}=f_{k+1}^{j_{k+1}}$
for $t$ the relevant stable letter (or the inverse thereof) associated to
the edge $e_k$. However we know that we only walk over
such an edge once by Proposition \ref{red}.

Let us start by supposing that $e_k$ is the first
edge walked over in $\Gamma\setminus T$.
The situation now is that we would already know the subgroup
\[\langle g_{k-1}\ldots g_1g_0a(g_{k-1}\ldots g_1g_0)^{-1},f_k\rangle\]
of $H$ has cohomological dimension 2 by the same argument as for the tree.
Now we can use our edge $e_k$ to form the HNN extension 
$G_1=\langle H,t\rangle$, although this time the edge subgroups will be
written as $\langle f_k\rangle\leq G_{v_{k-1}}$ and 
$\langle tf_kt^{-1}\rangle\leq G_{v_k}$, so that  
 the conjugate subgroup
$\langle tg_{k-1}\ldots g_1g_0a(g_{k-1}\ldots g_1g_0)^{-1}t^{-1},
tf_kt^{-1}\rangle$ has cohomological dimension 2 as well.

We also consider the subgroup
\[H_k=\langle tg_{k-1}\ldots g_1g_0a(g_{k-1}\ldots g_1g_0)^{-1}t^{-1},
tf_kt^{-1}, g_k^{-1}f_{k+1}g_k\rangle,\] 
of $G_1$ which is equal to
\[\langle tg_{k-1}\ldots g_1g_0a(g_{k-1}\ldots g_1g_0)^{-1}t^{-1},
tf_kt^{-1}\rangle*_{\langle tf_kt^{-1}\rangle}\langle tf_kt^{-1}, 
g_k^{-1}f_{k+1}g_k\rangle\] by Lemma \ref{subha} (ii) where
\[A=\langle f_k\rangle,\,B=\langle tf_kt^{-1}\rangle,\,
J=\langle g_{k-1}\ldots g_1g_0a(g_{k-1}\ldots g_1g_0)^{-1},
f_k\rangle,\] and $L=\langle tf_kt^{-1}, 
g_k^{-1}f_{k+1}g_k\rangle$, with $L$ also having 
cohomological dimension 2 because $g_kLg_k^{-1}$ is a subgroup
of $G_{v_k}$ with powers of its two generators equal by 
the second conjugacy inequality above.
Thus on applying Lemma \ref{co2}
again, we have that $H_k$ has cohomological dimension 2 and therefore
so does 
\[\langle g_ktg_{k-1}\ldots g_1g_0a(g_ktg_{k-1}\ldots g_1g_0)^{-1},
f_{k+1}\rangle\]
as it is conjugate in $G_1$ into a subgroup of $H_k$.

If this is the only edge in the conjugacy path that lies
outside $T$ then the remainder of the proof is as above, because we now walk
over the remaining edges $e_{k+1},\ldots ,e_n$ which all lie in $T$. If
however there are further edges outside $T$ then we can build up the
fundamental group $G$ from $H$ by a sequence of HNN extensions
which we perform in the order we walk over them (and then arbitrarily
for any edges left over). Our proof now works as before except that
a stable letter will appear within the product
$g_n\ldots g_1g_0$ whenever we walk over an edge in $\Gamma\setminus T$.
\end{proof}

We can now obtain the existence of Baumslag-Solitar subgroups by reducing
to known facts about groups of cohomological dimension 2.
\begin{co} \label{bscoh}
If $G(\Gamma)$ is a finite graph of groups with infinite cyclic
edge groups and
where every vertex group is
torsion free, balanced, and satisfies the cohomological condition then   
the fundamental group $G$ is not balanced exactly when $G(\Gamma)$
contains a complete non level conjugacy path, which is exactly when
$G$ contains a non Euclidean Baumslag Solitar subgroup.
\end{co}
\begin{proof} If $G$ is balanced then so are its subgroups and
by Theorem \ref{bal}, $G$ is not balanced exactly when we
have such a path. If so then consider the construction in this proof where
we add edges one by one to a maximal tree and let us stop on the first
occasion where the resulting fundamental group is not balanced. This group
will end up being a subgroup of $G$ so we now replace the final graph of groups
$G(\Gamma)$ with this one. Thus on taking $H(\Delta)$ to be the graph of
groups immediately before this edge was added,
we now have $G=\langle H,t\rangle$ and vertex elements $a,b\in H$ such that
$G$ is the HNN extension with base $H$ and $tat^{-1}=b$. But as this
edge lies in a complete non level conjugacy path, following the rest
of the path implies that we obtain $h\in H$ with $ha^ih^{-1}=b^j$ for
$i,j$ with $|i|\neq |j|$. Thus application of Theorem \ref{precor} with
$G(\Gamma)$ now equal to $H(\Delta)$ and $G$ now equal to $H$ tells us
that $\langle hah^{-1},b\rangle$ has cohomological dimension at most 2.
Now on replacing the stable letter $t$ with the alternative stable letter
$s=h^{-1}t$, we have that $G$ is also the HNN extension with base $H$
and associated subgroups $\langle a\rangle$ and $\langle h^{-1}bh\rangle$
with $s$ conjugating $a$ to $c=h^{-1}bh$. Thus by Lemma \ref{subha}
the subgroup $S=\langle a,c,s\rangle$ is also an HNN extension with
base $\langle a,c\rangle$ and the same
associated subgroups. As the base has cohomological dimension
2, so does this HNN extension $S$
in which $sa^js^{-1}=a^i$ holds for
$|i|\neq |j|$. This does not imply that $S=\langle s,a\rangle$ is isomorphic
to $BS(j,i)$ but as $I_S(a)=S$ which is
of cohomological dimension 2, the main result of \cite{krop} tells us that
$S$ is a generalised Baumslag-Solitar group and we see that 
$r=j/i\neq\pm 1$ is the image of the elliptic element $a$ under the modular
homomorphism of $S$. This implies by \cite{lev} Proposition 7.5 that
$S$ contains a subgroup isomorphic to $BS(m,n)$, where 
$m/n$ is the expression of $r$ in lowest terms, and so $BS(m,n)$ is non
Euclidean.

If instead the fundamental group $G$ contains a non Euclidean Baumslag-Solitar
subgroup then this subgroup and hence $G$ itself is not balanced, so we are
covered by Theorem \ref{bal}.
\end{proof}
Thus for instance we have that a finite graph of groups with free vertex
groups and infinite cyclic edge groups is balanced if and only if it
does not contain a non Euclidean Baumslag - Solitar group, which was not
explicitly stated in \cite{wsox}.

\section{Hyperbolic graphs of groups} 
In a torsion free hyperbolic group $H$, an element $h$ is called
maximal if whenever $h=a^i$ for $a\in H$ we have $i=\pm 1$. However for
these groups this condition is equivalent to $h$ generating its own 
centraliser (or even its own intersector).
Here centralisers are
always infinite cyclic and moreover (on removal of the identity) they
partition $H$ into infinite cyclic subgroups which are all maximal.
In this section we will nearly always be dealing with torsion free word
hyperbolic groups, but if not then saying an element is maximal will 
actually mean here that it generates its own centraliser.

We have the Bestvina - Feighn theorem
for hyperbolicity of
amalgamated free products with infinite cyclic edge groups:
\begin{lem} \label{bfam} (\cite{bf} Section 7, second Corollary)
Suppose that $G=A*_CB$ is an amalgamated free product where
$C=\langle c\rangle$ is infinite cyclic and $A,B$ are torsion
free word hyperbolic groups. Then $G$ is word hyperbolic if and
only if $c$ is maximal in one of $A$ or $B$, which occurs if and only
if $G$ does not contain $\Z^2$.
\end{lem}
The following result can be deduced directly from this along with
use of reduced forms, but by considering centralisers rather than
the maximal elements themselves.
\begin{lem} \label{mxlemam}
Suppose that $G=A*_CB$ is an amalgamated free product where
$C=\langle c\rangle$ is infinite cyclic and $A,B$ are torsion
free word hyperbolic groups. Let us take an element $a\in A$. Then\\
(i) If $a$ is not maximal in $A$ it is clearly not maximal in $G$.\\
(ii) If $a$ is maximal in $A$ but is not conjugate in $A$ into $C$ then
$a$ is still maximal in $G$.\\
Thus suppose from now on that $a$ is maximal in $A$ and conjugate in
$A$ into $C$, thus to exactly one of $c$ and $c^{-1}$ so say $c$ without
loss of generality. In particular $c$ is maximal in $A$ and thus $G$ is
word hyperbolic by Lemma \ref{bfam}.\\
(iii) Suppose that $c$ is also maximal in $B$
then $c$ is still maximal in $G$.\\
(iv) Suppose that $c$ is not maximal in $B$ then any element 
conjugate in $A$ to $c^{\pm 1}$ 
is clearly maximal in $A$ but not maximal in $G$.
\end{lem}
We now introduce a similar notion to the conjugacy paths already considered,
in order to keep track of vertex elements which are maximal in their
vertex group but not in the fundamental group.
Given an unoriented
edge $e$ in the graph of groups $G(\Gamma)$ with infinite cyclic
edge groups and torsion free hyperbolic vertex groups, we consider
the inclusions of the edge group $G_e$ in its neighbouring vertex
groups $G_v$ and $G_w$ (where maybe $v=w$). If $G_e$ is included in
$G_v$ as a non maximal subgroup then we put an arrow on $e$ at the
end of $e$ next to $v$ and we make this arrow point towards $v$. If
$G_e$ is included maximally in $G_v$ then no arrow is added, and we then
do the same with $G_w$ and then over all edges in the graph.
\begin{defn} 
Given vertex elements $g_v\in G_v$ and $g_w\in G_w$ (where maybe $v=w$), 
a {\bf semi non maximal path} from $g_v$ to $g_w$ is a conjugacy path
in $G(\Gamma)$ from $g_v$ to $g_w$
such that no edges in this path are
labelled with arrows as above, apart from a single arrow on the initial edge
which points towards the vertex $v$.

A {\bf full non maximal path} from $g_v$ to $g_w$ is a reduced
conjugacy path 
in $G(\Gamma)$ from $g_v$ to $g_w$
such that all edges in this path are
unlabelled, apart from an arrow on the initial edge
which points towards the vertex $v$ and an arrow on the final edge
which points towards $w$, and such that the final edge is not the reverse
of the initial edge.
\end{defn}
\begin{prop} \label{redfll}
Suppose we have a finite graph $G(\Gamma)$ of torsion free word hyperbolic
groups with infinite cyclic edge groups.
If there exists a semi non maximal path
from $g\in G_{v_0}$ to $g'\in G_{v_n}$ then
we can replace it with one that also runs from $g\in G_{v_0}$ to 
$g'\in G_{v_n}$ but which
traverses any unoriented edge at most once. The same statement is true
for any full maximal path. 
\end{prop}
\begin{proof}
We again run through the proof of Proposition \ref{red}, noting that the
initial edge of any semi maximal path is the only labelled edge, thus
this will remain and only unlabelled edges will be removed. Similarly
the only labelled edges in a full maximal path are the initial and final
ones, but these cannot be the same oriented edge as the arrows point
in different directions (unless the path is just this single edge) and the
definition rules out them being the reverse of each other. Thus again only
unlabelled edges can be removed. 
\end{proof}
We can use semi maximal paths to determine maximal elements in a graph of 
groups: suppose there is one from $g_v\in G_v$ to $g_w\in G_w$ where
$g_v$ and $g_w$ are both maximal in their respective vertex groups. 
If this path is considered just as a conjugacy path,
it would only indicate that some powers of $g_v$ and of $g_w$ are
conjugate in $G$. But as every edge apart from the initial one is
unmarked, all generators of these edge groups include into 
all subsequent vertex groups 
as maximal elements, and therefore we see that actually we have
$g_v^i$ and $g_w$ are conjugate in $G$ for $|i|>1$ so that
$g_w$ is not maximal in $G$, whether or not $g_v$ is.

We now consider when such a graph of groups
has a fundamental group which is word hyperbolic, starting with a tree.
\begin{thm} \label{hyptr}
Suppose that $G(\Gamma)$ is a graph of groups where $\Gamma$
is a tree, with infinite
cyclic edge groups and all vertex groups torsion free word hyperbolic.
Then $G$ is word hyperbolic unless 
there exists a full non maximal path in $G(\Gamma)$, in which case $G$ 
contains $\Z^2$, and if so then there exists such a path passing through
any non oriented edge at most once by Proposition \ref{redfll}
thus there are only finitely many
paths to check.
\end{thm}
\begin{proof}
As before, the proof is by induction on the number of edges. However
we also need to keep track of maximal elements so our inductive
statement is as follows:\\
\hfill\\
(i) If there exists a full non maximal path in $G(\Gamma)$ then $G$ 
contains $\Z^2$ and so is not word hyperbolic.\\
(ii) If $G(\Gamma)$ contains no full non maximal paths then $G$ is 
word hyperbolic.\\
(iii) If $G$ is word hyperbolic
then a vertex element $g_v$ which is
maximal in its vertex group is non maximal in $G$ if and only if there
exists a semi non maximal path that ends at $g_v$.\\
\hfill\\
Lemmas \ref{bfam} and \ref{mxlemam} give us the base case. 
Now given the graph of groups
$A(T)$, let us add an edge $e$ to
the tree $T$ to form the tree $\Gamma$ with $e$ having the vertex 
$v\in T$ and $w\notin T$, so that the fundamental group
$G$ of $G(\Gamma)$ is equal to $A*_CB$ where 
$B$ is the vertex group $G_w$.

By the inductive hypothesis applied to $A(T)$,
if there were a full non maximal path
in $A(T)$ and thus in $G(\Gamma)$
then $A$ and $G$ would contain $\Z^2$, so (i), (ii), (iii) hold
for $G(\Gamma)$ in this case.
Hence now we assume that no such path exists in $A(T)$
and consequently by induction
that $A$ is torsion free word hyperbolic.

First suppose that there is no arrow on $e$ pointing towards the vertex
$w$, so that $C=\langle c\rangle$ is included maximally in $B$.
Then no full maximal path can
lie in $G(\Gamma)$ that does not already lie in $A(T)$ (as these are
reduced paths by definition), so (i) holds
in this case. Moreover as we already have maximality on one side, we
obtain (ii) because $G$ is word hyperbolic by \cite{bf}.
As for (iii), on application of 
Lemma \ref{mxlemam} (with $A$ and $B$ as they are and then swapped),
we see that the only maximal elements of $A$ and $B$ that might no
longer be maximal in $G$ are those maximal
elements $b$ of $B$ which are conjugate in $B$ into $C$, and then only if
$c$ is not a maximal element of $A$. But this latter case can only
happen if either
$c$ is not a maximal element in the vertex group $G_v$, in
which case $e$ has an arrow pointing towards $v$ and we have our
semi non maximal path from $c\in G_v$ to $c\in G_w=B$, or by the
inductive hypothesis we have a semi non maximal path in $A(T)$ that starts
somewhere else and ends at $c\in G_v$, in which case we add the edge $e$
to the end of this path to get one that ends at $b\in B=G_w$.
Conversely the only way that new semi non maximal paths can be created in
$G(\Gamma)$ is by using $e$ itself, either as the only edge if it has an
arrow pointing towards $v$, or by adding it on if no arrow
is present. If this is not the final edge then we immediately need to
backtrack, thus Proposition \ref{redfll} tells us that we can regard it
as lying in $A(T)$ anyway.

For our second case,
we suppose that there is an arrow on $e$ pointing towards $w$, so
that in $B=G_w$ we have $c=d^i$ say, for $d$ maximal in $B$ and $|i|>1$.
If there is another arrow on $e$ pointing towards $v$ then we immediately
have a full non maximal path in $G(\Gamma)$, with $G$ 
non hyperbolic and containing $\Z^2$ by \cite{bf} because neither edge
inclusion is maximal. Otherwise $c$ is a maximal
element of $G_v$ which might not or might be a maximal element of $A$.
By part (iii) of the inductive hypothesis applied to $A$, which is
assumed to be word hyperbolic, this is determined by the existence
or not of a semi non maximal path in $A(T)$ which ends at $c\in G_v$
(assumed reduced by Proposition \ref{redfll}).
But if one exists then by adding the edge $e$ to the end, we have a
full non maximal path in $G(\Gamma)$ with $G=A*_CB$ being non maximal
on both sides, thus again we have $\Z^2$ in $G$ and (i) is confirmed
in this case too. If however there are no semi non maximal paths
in $A(T)$ that end at $c\in G_v$ then we can assume that $c$ is maximal
in $A$ and thus $G=A*_CB$ is word hyperbolic, giving (ii) here.

Finally we need to establish (iii) in the second case, whereupon
we are now dealing with $c$ being maximal in $A$ as well as in $G_v$. 
But $c$ is clearly not maximal in $G$ and so again by
Lemma \ref{mxlemam} the ``new non maximal vertex elements'' are exactly
those vertex elements that are maximal in $A$ and
are conjugate in $A$ to $c^{\pm 1}\in G_v$. Now
a conjugacy path in $A(T)$ from such an element $g_u\in G_u$ for $u\in T$ 
to $c\in G_v$ as in Section 7
is clearly a necessary condition for this to occur. Let us take
such a path (travelling over any edge at most once)
and first look at where arrows might appear on these edges. If we start
at $g_u\in G_u$ and walk towards $c\in G_w$, 
suppose that along the way we encounter an arrow
pointing in the opposite direction of travel. As we end up at $w$ with
an arrow pointing along with us, on finding the final arrow encountered
that points in the reverse direction
we obtain a full non maximal
path in $G(\Gamma)$, at least on application of Proposition \ref{redfll}.
But we have already established part (i) of
our inductive statement for $G(\Gamma)$ so this is a contradiction.

If however we encounter an arrow along the way that points in our direction 
then the reverse of this path so far is already 
a semi non maximal path lying
in $A(T)$ and ending at $g_u$, so by the inductive hypothesis $g_u$ was
not maximal in $A$.
Thus our conjugacy path has no arrows appearing and thus by adding on 
the edge $e$ at the end and reversing, we
are left with a semi non maximal path running
from $c\in B$ to $g_u\in G_u$. As for the converse, the new semi non maximal
paths will all be created by starting at $b\in G_w$ with the edge $e$ and
then following a conjugacy path in $A(T)$ which has no edges labelled,
which certainly means that the end element is not maximal in $G$.
\end{proof} 

We now move to HNN extensions, where \cite{bf2} gives
the exact conditions for word hyperbolicity over virtually cyclic
groups, though again we only state it here in the torsion free case.
\begin{prop} \label{bfhn} (\cite{bf2} Corollary 2.3)
Let $H$ be a torsion free hyperbolic group and let us form the HNN
extension $G=\langle H,t\rangle$ over the
infinite cyclic subgroups $A=\langle a\rangle$ and
$B=\langle b\rangle$ where $tat^{-1}=b$. Then $G$ is word hyperbolic
unless\\
Either: some conjugate of $B$ in $H$ intersects $A$ non trivially\\
Or: The intersector $I_H(a)\neq A$ and the intersector $I_H(b)\neq B$.\\
In both of these cases $G$ contains a Baumslag - Solitar 
subgroup and so is not word hyperbolic.
\end{prop}

We will again need a version of Lemma \ref{mxlemam} for HNN extensions
(proved in the same manner using reduced forms):
\begin{lem} \label{mxlemhn}
Let  $G=\langle H,t\rangle$ be the HNN extension of
the torsion free hyperbolic group $H$ over the
infinite cyclic subgroups $A=\langle a\rangle$ and
$B=\langle b\rangle$ where $tat^{-1}=b$. Then on taking an element
$h\in H$ we have:\\
(i) If $h$ is not maximal in $H$ then it is clearly not maximal in $G$.\\
(ii) If $h$ is maximal in $H$ and is not conjugate in $H$ into
$A\cup B$ then $h$ is still maximal in $G$.\\
Now suppose that no conjugate of $B$ in $H$ intersects $A$ non trivially.
Suppose also that $h$ is maximal in $H$ and without loss of
generality is conjugate in $H$
into $A$, thus to exactly one of $a$ and $a^{-1}$
so say $a$. In particular $a$ is maximal in $H$ and thus $G$ is word
hyperbolic by Proposition \ref{bfhn}.\\ 
(iii) Suppose that $b$ is also maximal in $H$ then 
$a$ and $b$ are maximal in $G$.\\
(iv) Suppose that $b$ is not maximal in $H$ then clearly any element conjugate
in $H$ to $a^{\pm 1}$ is maximal in $H$ but not in $G$ and any element
conjugate in $H$ to $b^{\pm 1}$ is not maximal in $H$ nor in $G$. 
\end{lem}
We can now give our final result.  
\begin{thm} \label{hyp}
Suppose that $G(\Gamma)$ is a finite
graph of groups where all edge 
groups are infinite cyclic and all vertex groups are
torsion free word hyperbolic.
Then $G$ is word hyperbolic unless\\
(i) there exists a complete conjugacy path in  $G(\Gamma)$ or\\
(ii) there exists a full non maximal path in $G(\Gamma)$\\
in which case $G$ contains a Baumslag - Solitar group and so is not
word hyperbolic. If either of these hold then by Proposition \ref{redcmp}
for (i) and Proposition \ref{redfll} for (ii) we can assume that
such a path passes through any unoriented edge at most once and so there
are only finitely many of these paths to check.
\end{thm}
\begin{proof}
This reduces to Theorem \ref{hyptr}
if $\Gamma$ is a tree as $G(\Gamma)$ cannot then contain a
reduced closed conjugacy path. Otherwise we take a maximal tree $T$
and we assume by induction on the
number of edges in $\Gamma\setminus T$ that:\\ 
\hfill\\
(i) If there exists a complete conjugacy path in $G(\Gamma)$
then $G$ contains a Baumslag-Solitar group and so is not word hyperbolic.\\
(ii) If there exists a full non maximal path in $G(\Gamma)$ then $G$ 
contains a Baumslag - Solitar group and so is not word hyperbolic.\\
(iii) If $G(\Gamma)$ contains no full non maximal paths and no complete
conjugacy paths then $G$ is word hyperbolic.\\
(iv) If $G$ is word hyperbolic
then a vertex element $g_v$ which is
maximal in its vertex group is non maximal in $G$ if and only if there
exists a semi non maximal path that ends at $g_v$.\\
\hfill\\
Thus we assume we have the graph of groups $H(\Delta)$ satisfying these
conditions and we add an edge $e$ to $\Delta$ to obtain $\Gamma$. The base
case where $\Gamma$ is a tree is covered by Theorem \ref{hyptr} (or
$\Gamma$ is a single point whereupon everything automatically holds).
In general we can assume that $H(\Delta)$ is word hyperbolic
because otherwise the inductive hypothesis applied to $H(\Delta)$ tells us
that $H$ and thus $G$ would contain a Baumslag - Solitar group. In particular
there are no complete conjugacy paths
or full non maximal paths in $H(\Delta)$. On adding the edge $e$ to
$H(\Delta)$ to obtain $G(\Gamma)$, where $G$ will be the HNN extension
$\langle H,t\rangle$ for $A=\langle a\rangle$ a cyclic subgroup of
$G_v$ and $B=\langle b\rangle$ of $G_w$, we suppose that 
$C=\langle c\rangle$ is the maximal cyclic subgroup of $G_v$ containing
$a$ and $D=\langle d\rangle$ that for $b$ in $G_w$.

First suppose that addition of $e$ creates a complete conjugacy
path in $G(\Gamma)$. This means that there must have been a reduced
conjugacy path from $a$ to $b$ lying in $H(\Delta)$, or possibly that
$a$ and $b$ lie in the same vertex group with powers conjugate in that
group, so our HNN extension
fails the ``Either'' condition for hyperbolicity in 
Proposition \ref{bfhn}. Hence
$G$ contains a Baumslag - Solitar subgroup and so
is not word hyperbolic (nor any group obtained from $G$ by further HNN
extensions). 
Alternatively if on addition of $e$ we find a full non maximal path
in $G(\Gamma)$, when there were none in $H(\Delta)$, then $e$ might be
labelled by arrows: if there are two arrows on $e$ then right away $G$
satisfies the ``Or'' condition in Proposition \ref{bfhn} and so contains
a Baumslag - Solitar group. If $e$ has one arrow then suppose that $e$  joins 
$a=g_v\in G_v$ to $b=g_w\in G_w$ (where possibly $v=w$) and $b$ is not
maximal in $G_w$ so our arrow on $e$ points towards $w$. But now by
starting at $v$ and then following the rest of our full non maximal path,
we have that the reverse of this is a semi non maximal path which must
lie purely within $H(\Delta)$ and so by (iv) applied to $H(\Delta)$
we have that $a\in G_v$ is not maximal in $H$ either and so $G$ contains
a Baumslag - Solitar group by Proposition \ref{bfhn}. Similarly if $e$ is
unlabelled then we can do the same in both directions from $v$ and from $w$.

For (iii), suppose that $G$ fails to be word hyperbolic because the
``Either'' condition holds in Proposition \ref{bfhn}. Then by Theorem
\ref{cnjgr} applied to $H(\Delta)$
we have that either $v=w$ and a conjugate within $G_v$
of $B$ meets $A$, in which case the loop $e$ itself is a complete
conjugacy path, or there is a conjugacy path from $a\in G_v$ to 
$b\in G_w$ which can be made reduced by Proposition \ref{red} and thus
is made complete by adding on $e$.

Thus now the ``Or'' condition in Proposition \ref{bfhn} is the only way
that $G$ can fail to be word hyperbolic. So
suppose that both $a$ and $b$ fail to be maximal
in $H$. It could be that $a$ and $b$ are not maximal in each of their
respective vertex groups, in which case $e$ is marked with two arrows and
itself immediately provides a full non maximal path in $G(\Gamma)$. 
Alternatively it
could be that $a$ (say) is maximal in $G_v$ but not in $H$ whereas
$b$ is not even maximal in $G_w$, whereupon we use the inductive 
hypothesis (iv) on $H(\Delta)$ to get us a semi non maximal path ending
in $a\in G_v$, reduced without loss of generality,
which results in a full non maximal path when putting $e$
on the end. If however both $a$ and $b$ are maximal in their respective
vertex groups then the edge $e$ is unlabelled and we have two reduced
semi non
maximal paths, each in $H(\Delta)$ with one ending in $a\in G_v$ and the
other in $b\in G_w$, thus by putting these together with $e$ in the
middle we again have a full non maximal path in $G(\Gamma)$.

Finally we must establish the inductive hypothesis (iv) for $G(\Gamma)$ by
determining
which are the ``new non maximal vertex elements''.
First say that both $a$ and $b$ are maximal in $H$ then we have 
by Lemma \ref{mxlemhn} that there are none. But no new semi non
maximal paths can be created using the edge $e$, which is
unlabelled, because following one from its start until we reach $e$
would produce a semi non maximal path in $H(\Delta)$ which ended in
either $a$ or $b$, thus by the inductive hypothesis one of these elements
would be non maximal in $H$.

So one of $a$ or $b$ is non maximal in $H$, but both being
non maximal puts us in the non word hyperbolic case, thus
we now suppose without loss of generality that $b$ is non
maximal in $H$ but $a$ is maximal. Then again by Lemma \ref{mxlemhn}
the new non maximal vertex elements
will be those that are conjugate in $H$ to $a^{\pm 1}\in G_v$, thus 
we can take our semi non maximal path in $H(\Delta)$ that ends in $b\in G_w$ 
and add the edge $e$ (or just take the edge $e$ if $b$ is not maximal in
$G_w$) to get one that ends in $a\in G_v$. If we further have maximal
elements in $H$ that are conjugate in $H$ but not in $G_v$ to $a$
(or to $a^{-1}$ in which case we replace every element in the path by its
inverse) then we have a conjugacy path from such an element to $a\in G_v$.
We can assume that no arrows appear on any edges in this path because
otherwise we either lose maximality of this element
in $H(\Delta)$ or we create
a full non maximal path in $G(\Gamma)$, just as in the proof of Theorem
\ref{hyptr}. Thus by first following the semi maximal path to $b$,
then the reverse of $e$ and finally the reverse of this conjugacy path
to $a$, we obtain our
semi non maximal path to any of these ``new non maximal
vertex elements''. Moreover this is the only way in which new 
semi non maximal paths can be created
when moving from $H(\Delta)$ to $G(\Gamma)$ because, once we assume such
a path only passes through each non oriented edge at most once,
it would have to travel through $e$ from $b$ to $a$ (else it would lie
in $H(\Delta)$ or the element $a$ would not be maximal in $H$), 
thus our induction is complete.
\end{proof}

\Address

\end{document}